\documentclass{amsart}

\usepackage{amscd}
\usepackage{amssymb}

\usepackage{color}

\usepackage[all]{xy}
\usepackage{amsthm}

\newtheorem{thm}[equation]{Theorem}
\newtheorem{lem}[equation]{Lemma}
\newtheorem{cor}[equation]{Corollary}
\newtheorem{prop}[equation]{Proposition}
\newtheorem{conj}[equation]{Conjecture}

\newtheorem*{thm*}{Theorem}
\newtheorem*{prop*}{Proposition}
\newtheorem*{cor*}{Corollary}
\newtheorem*{lem*}{Lemma}
\newtheorem*{MT*}{Main Theorem}
\newtheorem{ques}[equation]{Question}
\newtheorem*{ques*}{Question}

\newtheorem{specthm}{Theorem}

\theoremstyle{definition} %

\newtheorem*{defn*}{Definition}

\newtheorem{eg}[equation]{Example}

\theoremstyle{remark} %
\newtheorem{rmk}[equation]{Remark}

\newtheorem*{rmk*}{Remark}
\newtheorem*{rmks*}{Remarks}

\newtheoremstyle{exercise}% name
  {3pt}%      Space above
  {3pt}%      Space below
  {\small}%         Body font
  {}%         Indent amount (empty = no indent, \parindent = para indent)
  {\sc\small}% Thm head font
  {.}%        Punctuation after thm head
  {.5em}%     Space after thm head: " " = normal interword space;
   {}     %       \newline = linebreak
  {}%         Thm head spec (can be left empty, meaning `normal')

\theoremstyle{exercise}

\makeatletter
{\renewcommand{\theequation}{#1}}%
{\renewcommand{\theequation}{\arabic{equation}}\addtocounter{equation}{-1}\global\@ignoretrue}
\makeatother

\makeatletter
{\renewcommand{\theequation}{#1}\begin{eqnarray}}%
{\end{eqnarray}\renewcommand{\theequation}{\arabic{equation}}\addtocounter{equation}{-1}\global\@ignoretrue}
\makeatother

\makeatletter
{\smallskip \refstepcounter{equation}\noindent{\textbf{\theequation.} }{{\textbf{#1.}}}}%
{\smallskip \global\@ignoretrue}
\makeatother

\makeatletter
{\smallskip \refstepcounter{equation}\noindent{\textbf{\theequation.} }{{\textbf{#1}}}}%
{\smallskip \global\@ignoretrue}
\makeatother

\makeatletter
{\smallskip \refstepcounter{equation}{\sc \theequation}{\sc (#1).}}%
{\smallskip \global\@ignoretrue}
\makeatother

\makeatletter
{\smallskip \refstepcounter{equation}\noindent{\sc \theequation.}{\sl{ #1.}}}%
{\smallskip \global\@ignoretrue}
\makeatother

\makeatletter
\newenvironment{borel*}%
{\smallskip \refstepcounter{equation}\noindent{\textbf{\theequation.}}}%
{\global\@ignoretrue}
\makeatother

\newcommand{\flist}[1]{\hangindent\leftmargini\textup{(1)}\hskip\labelsep {#1}%
\begin{enumerate}%
\setcounter{enumi}{1}%
}

\newcommand{\ot}{\otimes}

\newcommand{\R}{{\mathbb{R}}}        % the reals
\newcommand{\Z}{{\mathbb{Z}}}        % the integers
\newcommand{\D}{\Delta}

\newcommand{\oddots}{{\mathinner{\mkern1mu\raise1pt\vbox{\kern7pt\hbox{.}}\mkern2mu\raise4pt\hbox{.}\mkern2mu\raise7pt\hbox{.}\mkern1mu}}}

\newcommand{\Bt}{{(B,\tau)}}

\newcommand{\Fsep}{F_{{\mathrm{sep}}}}

\newcommand{\Fx}{{F^{\times}}}

\DeclareMathOperator{\Spin}{Spin}           % the Spin group

\DeclareMathOperator{\SL}{SL}
\DeclareMathOperator{\GL}{GL}
\DeclareMathOperator{\PGL}{PGL}
\DeclareMathOperator{\PSO}{PSO}

\newcommand{\SO}{\mathrm{SO}}
\newcommand{\SU}{\mathrm{SU}}

\newcommand{\Gm}{\mathbb{G}_m}

\DeclareMathOperator{\Gal}{Gal}

\DeclareMathOperator{\im}{im}

\DeclareMathOperator{\car}{char}

\DeclareMathOperator{\Id}{Id}

\DeclareMathOperator{\cores}{cor}

\DeclareMathOperator{\diag}{diag}

\DeclareMathOperator{\End}{End}

\DeclareMathOperator{\aut}{Aut}
\DeclareMathOperator{\Aut}{Aut}

\newcommand{\Hom}{{\mathrm{Hom}}}

\newcommand{\Nrd}{{\mathrm{Nrd}}}

\newcommand{\iso}{\xrightarrow{\sim}}

\usepackage{hyperref}
\hypersetup{
  colorlinks=true, %set true if you want colored links
   linkcolor=red,  %choose some color if you want links to stand out
}

\usepackage[hyphenbreaks]{breakurl}

\newcommand{\Falg}{F_{{\mathrm{alg}}}}

\DeclareMathOperator{\Sym}{Sym}
\DeclareMathOperator{\Str}{Str}
\DeclareMathOperator{\Diag}{Diag}

\newcommand{\Eins}{\mathbf{1}}

\newcommand{\IZ}{\mathbb{Z}}
\newcommand{\bft}{\mathbf{t}}

\newcommand{\mfV}{\mathfrak{V}}

\newcommand{\Her}{\mathop\mathrm{Her}\nolimits}
\newcommand{\Symd}{\mathop\mathrm{Symd}\nolimits}
\newcommand{\Mat}{\mathop\mathrm{Mat}\nolimits}
\newcommand{\ch}{\mathop \mathrm{char}\nolimits}

\newcommand{\Zor}{\mathop \mathrm{Zor}\nolimits}
\newcommand{\kalg}{k\mathchar45\mathbf{alg}}

\renewcommand{\:}{\colon\,}

\numberwithin{equation}{subsection}

\renewcommand{\theenumi}{\roman{enumi}}

\begin{document}

\title[Outer automorphisms and a Skolem-Noether theorem]{Outer automorphisms of algebraic groups and a Skolem-Noether theorem for Albert algebras}
\author{Skip Garibaldi}
\address{Garibaldi: Institute for Pure and Applied Mathematics, UCLA, 460 Portola Plaza, Box 957121, Los Angeles, California 90095-7121, USA}
\email{skip@member.ams.org}
%\url{http://www.mathcs.emory.edu/{\textasciitilde}skip}

\author{Holger P. Petersson}
\address{Petersson: Fakult\"at f\"ur Mathematik und Informatik, FernUniversit\"at in Hagen, D-58084 Hagen, Germany}
\email{holger.petersson@fernuni-hagen.de}

%\date{\tt Version of \today}

\dedicatory{Dem Andenken Reinhard B\"orgers gewidmet}

\subjclass[2010]{Primary 20G41; Secondary 11E72, 17C40, 20G15}

\begin{abstract}
The question of existence of outer automorphisms of a simple algebraic group $G$ arises naturally both when working with the Galois cohomology of $G$ and as an example of the algebro-geometric problem of determining which connected components of $\Aut(G)$ have rational points.  The existence question remains open only for four types of groups, and we settle one of the remaining cases, type $^3D_4$. The key to the proof is a Skolem-Noether theorem for cubic \'etale subalgebras of Albert algebras which is of independent interest. Necessary and sufficient conditions for a simply connected group of outer type $A$ to admit outer automorphisms of order 2 are also given.
\end{abstract}

\maketitle

\setcounter{tocdepth}{1}

\tableofcontents

%%%%%%%%%%%%%%%%%%%%%%%%%%%%%%%%%%%%%%%%%

\section{Introduction}

An algebraic group $H$ defined over an algebraically closed field $F$ is a disjoint union of connected components.  The component $H^\circ$ containing the identity element is a normal subgroup in $H$ that acts via multiplication on each of the other components.  Picking an $F$-point $x$ in a connected component $X$ of $H$ gives an isomorphism of varieties with an $H^\circ$-action $H^\circ \iso X$ via $h \mapsto hx$.

When $F$ is not assumed to be algebraically closed, the identity component $H^\circ$ is still defined as an $F$-subgroup of $H$, but the other components need not be.  Suppose $X$ is a connected subvariety of $H$ such that, after base change to the algebraic closure $\Falg$ of $F$, $X \times \Falg$ is a connected component of $H \times \Falg$.  Then, by the previous paragraph, $X$ is an $H^\circ$-torsor, but $X$ may have no $F$-points.  We remark that the question of whether $X$ has an $F$-point  arises when describing the embedding of the category of compact real Lie groups into the category of linear algebraic groups over $\R$, see \cite[\S5]{Se:Gebres}.

\subsection{Outer automorphisms of algebraic groups}
We will focus on the case where $H = \Aut(G)$ and $G$ is semisimple, which amounts to asking about the existence of outer automorphisms of $G$.  This question has previously been studied in  \cite{MT}, \cite{PreetiTig},  \cite{G:outer}, \cite{ChKT},  \cite{ChEKT}, and \cite{KT:3D4}.  Writing $\Delta$ for the Dynkin diagram of $G$ endowed with the natural action by the Galois group $\Gal(\Fsep/F)$ gives an exact sequence of group schemes 
\[
\begin{CD}
1 @>>> \aut(G)^\circ @>>> \aut(G) @>\alpha>> \Aut(\D) 
\end{CD}
\]
as in \cite[Chap.~XXIV, Th.~1.3 and \S3.6]{SGA3} or \cite[\S16.3]{Sp:LAG}, hence a natural map $\alpha(F) \!: \aut(G)(F) \to \aut(\D)(F)$.  Note that $\aut(\D)(\Falg)$ is identified with the connected components of $\Aut(G) \times \Falg$ in such a way that $\aut(\D)(F)$ is identified with those components  that are defined over $F$.  We ask: is $\alpha(F)$ onto?  That is, \emph{which of the components of $\aut(G)$ that are defined over $F$ also have an $F$-point?}  

Sending an element $g$ of $G$ to conjugation by $g$ defines a surjection $G \to \aut(G)^\circ$, and the $F$-points $\aut(G)^\circ(F)$ are called \emph{inner} automorphisms. The $F$-points of the other components of $\aut(G)$ are called \emph{outer}.  Therefore, our question may be rephrased as: \emph{Is every automorphism of the Dynkin diagram induced from an $F$-automorphism of $G$?}  

One can quickly observe that $\alpha(F)$ need not be onto, for example, with the group $\SL(A)$ where $A$ is a central simple algebra of odd exponent, where an outer automorphism would amount to an isomorphism of $A$ with its opposite algebra.  This is a special case of a general cohomological obstruction.  Namely, writing $Z$ for the scheme-theoretic center of the simply connected cover of $G$, $G$ naturally defines an element $t_G \in H^2(F, Z)$ called the \emph{Tits class} as in \cite[4.2]{Ti:R} or \cite[31.6]{KMRT}.  (The cohomology used in this paper is fppf.) For every character $\chi \!: Z \to \Gm$, the image $\chi(t_G) \in H^2(F, \Gm)$ is known as a Tits algebra of $G$; for example, when $G = \SL(A)$, $Z$ is identified with the group of $(\deg A)$-th roots of unity, the group of characters is generated by the natural inclusion $\chi \!: Z \hookrightarrow \Gm$, and $\chi(t_{\SL(A)})$ is the class of $A$.  (More such examples are given in \cite[\S27.B]{KMRT}.) This example illustrates also the general fact: $t_G = 0$ if and only if $\End_G(V)$ is a field for every irreducible representation $V$ of $G$.  The group scheme $\Aut(\D)$ acts on $H^2(F, Z)$, and it was shown in \cite[Th.~11]{G:outer} that this provides an obstruction to the surjectivity of $\alpha(F)$, namely:
\begin{equation} \label{alpha}
\im \left[\alpha \!: \Aut(G)(F) \to \Aut(\D)(F) \right] \subseteq \{ \pi \in \Aut(\D)(F) \mid \pi(t_G) = t_G \}.
\end{equation}

It is interesting to know when equality holds in \eqref{alpha}, because this information is useful in Galois cohomology computations.  (For example, when $G$ is simply connected, equality in \eqref{alpha} is equivalent to the exactness of $H^1(F, Z) \to H^1(F, G) \to H^1(F, \Aut(G))$.) Certainly, equality need not hold in \eqref{alpha}, for example when $G$ is semisimple (take $G$ to be the product of the compact and split real forms of $G_2$) or when $G$ is neither simply connected nor adjoint (take $G$ to be the split group $\SO_8$, for which $|{\im \alpha}| = 2$ but the right side of \eqref{alpha} has 6 elements).  However, when $G$ is simple and simply connected or adjoint, it is known that equality holds in \eqref{alpha} when $G$ has inner type or for some fields $F$. Therefore, one might optimistically hope that the following is true: 

\begin{conj} \label{outer.conj}
If $G$ is an absolutely simple algebraic group that is simply connected or adjoint, then equality  holds in \eqref{alpha}.
\end{conj}

The remaining open cases are where $G$ has type $^2A_n$ for odd $n \ge 3$ (the case where $n$ is even is Cor.~\ref{conj.A}), $^2D_n$ for $n \ge 3$, $^3D_4$, and $^2E_6$.
Most of this paper is dedicated to settling one of these four cases.

\begin{specthm} \label{MT.D4}
If $G$ is a simple algebraic group of type $^3D_4$ over a field $F$, then equality holds in \eqref{alpha}.
\end{specthm}

One can ask also for a stronger property to hold:

\begin{ques} \label{question.refined}
Suppose $\pi$ is in $\alpha(\aut(G)(F))$.  Does there exist a $\phi \in \Aut(G)(F)$ so that $\alpha(\phi) = \pi$ and $\phi$ and $\pi$ have the same order?
\end{ques}

This question, and a refinement of it where one asks for detailed information about the possible $\phi$'s, was considered for example in \cite{MT}, \cite{PreetiTig}, \cite{ChKT},  \cite{ChEKT}, and \cite{KT:3D4}.  It was observed in \cite{G:outer} that the answer to the question is ``yes'' in all the cases where the conjecture is known to hold.  However, \cite{KT:3D4} gives an example of a group $G$ of type $^3D_4$ that does not have an outer automorphism of order 3, yet the conjecture holds for $G$ by Theorem \ref{MT.D4}.  That is, combining the results of this paper and \cite{KT:3D4} gives the first example where the conjecture holds for a group but the answer to Question \ref{question.refined} is  ``no'', see Example \ref{KT.counter}

In other sections of the paper, we translate the conjecture for groups of type $A$ into one in the language of algebras with involution as in \cite{KMRT}, give a criterion for the existence of outer automorphisms of order 2 (i.e., prove a version for type $A$ of the main result of \cite{KT:3D4}), and exhibit a group of type $^2A$ that does not have an outer automorphism of order 2.  % We repeat these same kinds of results for groups of type $^2E_6$.

%%%%%%%%%%%%%%%%%%%%%%%%%%%%%%%%%%%%%%%%%%%%%%%%%%%%%%%%%%%%%%%%%%%%%%%%%%%%%%
\subsection{Skolem-Noether Theorem for Albert algebras} \label{ss.SKONOAL}  
In order to prove Theorem \ref{MT.D4}, we translate it into a statement about Albert $F$-algebras, 27-dimensional exceptional central simple Jordan algebras.  We spend the majority of the paper working with Jordan algebras.

Let $J$ be an Albert algebra over a field $F$ and suppose $E,E^\prime \subseteq J$ are cubic \'etale subalgebras. It is known since Albert-Jacobson \cite{MR0088487} that in general an isomorphism $\varphi\:E \to E^\prime$ cannot be extended to an automorphism of $J$. Thus the Skolem-Noether Theorem fails to hold for cubic \'etale subalgebras of Albert algebras. In fact, even in the important special case that $E = E^\prime$ is split and $\varphi$ is an automorphism of $E$ having order $3$, obstructions to the validity of this result may be read off from \cite[Th.~9]{MR0088487}. We provide a way out of this impasse by replacing the automorphism group of $J$ by its structure group and allowing the isomorphism $\varphi$ to be twisted by the right multiplication of a norm-one element in $E$. More precisely, referring to our notational conventions in Sections~\ref{s.CONV}$\--$\ref{s.CUJO} below, we will establish the following result.

\begin{specthm} \label{t.SNOTWI} {Let $\varphi\:E \overset{\sim} \to E^\prime$ be an isomorphism of cubic \'etale subalgebras of an Albert algebra $J$ over a field $F$. Then there exists an element $w \in E$ satisfying $N_E(w) = 1$ such that $\varphi \circ R_w\:E \to E^\prime$ can be extended to an element of the structure group of $J$.} 
\end{specthm}

 Note that no restrictions on the characteristic of $F$ will be imposed. In order to prove Theorem \ref{t.SNOTWI}, we first derive its analogue (in fact, a substantial generalization of it, see Th.~\ref{t.SKONOIT} below) for absolutely simple Jordan algebras of degree $3$ and dimension $9$ in place of $J$. This generalization is based on the notions of weak and strong equivalence for isotopic embeddings of cubic \'etale algebras into cubic Jordan algebras (\ref{ss.WESO}) and is derived here by elementary manipulations of the two Tits constructions. After a short digression into norm classes for pairs of isotopic embeddings in \S~6, Theorem \ref{t.SNOTWI} is established by combining Th.~\ref{t.SKONOIT} with a density argument and the fact that an isotopy between absolutely simple nine-dimensional subalgebras of an Albert algebra can always be extended to an element of its structure group (Prop.~\ref{p.SKONOIN}). 
 
% % % % % % % % % % % % % % % % % % % % % % % % % % % % % % % % % % % % % % % % % % % % % % % % % % %

\subsection{Conventions.} \label{s.CONV}  Throughout this paper, we fix a base field $F$ of arbitrary characteristic. All linear non-associative algebras (in particular, all composition algebras) are tacitly assumed to contain an identity element. If $C$ is such an algebra, we write $R_v\:C \to C$ for the right multiplication by $v \in C$, and $C^ \times$ for the collection of invertible elements in $C$, whenever this makes sense. For a field extension (or any commutative associative algebra) $K$ over $F$, we denote by $C_K := C \otimes K$ the scalar extension (or base change) of $C$ from $F$ to $K$, unadorned tensor products always being taken over $F$. In other terminological and notational conventions, we mostly follow  \cite{KMRT}. 
In fact, the sole truly significant deviation from this rule is presented by the theory of Jordan algebras: while \cite[Chap.~IX]{KMRT} confines itself to the linear version of this theory, which works well only over fields of characteristic not $2$ or, more generally, over commutative rings containing $\frac{1}{2}$, we insist on the quadratic one, surviving as it does in full generality over arbitrary commutative rings. For convenience, we will assemble the necessary background material in the next two sections of this paper.

% % % % % % % % % % % % % % % % % % % % % % % % % % % % % % % % % % % % % % % % % % % % % % % % % % % %

\section{Jordan algebras} \label{s.JOAL}  The purpose of this section is to present a dictionary for the standard vocabulary of arbitrary Jordan algebras. Our main reference is \cite{MR634508}.

\subsection{The concept of a Jordan algebra} \label{ss.COJA} By a (unital quadratic) \emph{Jordan algebra} over $F$, we mean an $F$-vector space $J$ together with a quadratic map $x \mapsto U_x$ from $J$ to $\End_F(J)$ (the \emph{$U$-operator}) and a distinguished element $1_J \in J$ (the \emph{unit} or \emph{identity element}) such that, writing
\[
\{xyz\} := V_{x,y}z := U_{x,z}y := (U_{x+z} - U_x - U_z)y
\]
for the associated \emph{triple product}, the equations
\begin{align}
U_{1_J} =\,\,&\Eins_J, \notag \\
\label{FUFO} U_{U_xy} =\,\,&U_xU_yU_x &&(\text{fundamental formula}), \\
U_xV_{y,x} =\,\,&V_{x,y}U_x \notag
\end{align}
hold in all scalar extensions. We always simply write $J$ to indicate a Jordan algebra over $F$, $U$-operator and identity element being understood. A \emph{subalgebra} of $J$ is an $F$-subspace containing the identity element and stable under the operation $U_xy$; it is then a Jordan algebra in its own right. A \emph{homomorphism} of Jordan algebras over $F$ is an $F$-linear map preserving $U$-operators and identity elements. In this way we obtain the category of Jordan algebras over $F$. By definition, the property of being a Jordan algebra is preserved by arbitrary scalar extensions. In keeping with the conventions of Section~\ref{s.CONV}, we write $J_K$ for the base change of $J$ from $F$ to a field extension $K/F$.

\subsection{Linear Jordan algebras} \label{ss.LIJO} Assume $\ch(F) \neq 2$. Then Jordan algebras as defined in \ref{ss.COJA} and linear Jordan algebras as defined in \cite[\S~37]{KMRT} are virtually the same. Indeed, let $J$ be a unital quadratic Jordan algebra over $F$. Then $J$ becomes an ordinary non-associative $F$-algebra under the multiplication $x\cdot y := \frac{1}{2}U_{x,y}1_J$, and this $F$-algebra is a linear Jordan algebra in the sense that it is commutative and satisfies the \emph{Jordan identity} $x\cdot ((x\cdot x)\cdot y) = (x\cdot x)\cdot (x\cdot y)$. Conversely, let $J$ be a linear Jordan algebra over $F$. Then the $U$-operator $U_xy := 2x\cdot (x\cdot y) - (x\cdot x)\cdot y$ and the identity element $1_J$ convert $J$ into a unital quadratic Jordan algebra. The two constructions are inverse to one another and determine an isomorphism of categories between unital quadratic Jordan algebras and linear Jordan algebras over $F$.

\subsection{Ideals and simplicity} \label{ss.IDSI} Let $J$ be a Jordan algebra over $F$. A subspace $I \subseteq J$ is said to be an \emph{ideal} if $U_IJ + U_JI + \{IIJ\} \subseteq J$. In this case, the quotient space $J/I$ carries canonically the structure of a Jordan algebra over $F$ such that the projection $J \to J/I$ is a homomorphism. A Jordan algebra is said to be \emph{simple} if it is non-zero and there are no ideals other than the trivial ones. We speak of an \emph{absolutely simple} Jordan algebra if it stays simple under all base field extensions. (There is also a notion of central simplicity which, however, is weaker than absolute simplicity, although the two agree for $\ch(F) \neq 2$.)

\subsection{Standard examples} \label{ss.STEX} First, let $A$ be an associative $F$-algebra. Then the vector space $A$ together with the $U$-operator $U_xy := xyx$ and the identity element $1_A$ is a Jordan algebra over $F$, denoted by $A^+$. If $A$ is simple, then so is $A^+$ \cite[15.5]{MR946263}.

Next, let $(B,\tau)$ be an $F$-algebra with involution, so $B$ is a non-associative algebra over $F$ and $\tau\:B \to B$ is an $F$-linear anti-automorphism of period $2$. Then
\begin{align*}
\Symd(B,\tau) = \{x + \tau(x)\mid x \in B\} \subseteq H(B,\tau) := \Sym(B,\tau) = \{x \in B \mid \tau(x) = x\} 
\end{align*}
are subspaces of $B$, and we have $\Symd(B,\tau) = H(B,\tau)$ for $\ch(F) \neq 2$ but not in general. Moreover, if $B$ is associative, then $\Symd(B,\tau)$ and $H(B,\tau)$ are both subalgebras of $B^+$, hence are Jordan algebras which are simple if $(B,\tau)$ is simple as an algebra with involution \cite[15.5]{MR946263}.

\subsection{Powers} \label{ss.PO} Let $J$ be a Jordan algebra over $F$. The powers of $x \in J$ with integer exponents $n \geq 0$ are defined recursively by $x^0 = 1_J$, $x^1 = x$, $x^{n+2} = U_xx^n$. Note for $J = A^+$ as in \ref{ss.STEX}, powers in $J$ and in $A$ are the same. For $J$ arbitrary, they satisfy the relations
\begin{align}
\label{PORU} U_{x^m}x^n = x^{2m+n}, \quad \{x^mx^nx^p\} = 2x^{m+n+p}, \quad (x^m)^n = x^{mn},
\end{align}
hence force
\[
F[x] := \sum_{n\geq 0} Fx^n 
\]
to be a subalgebra of $J$. In many cases --- e.g., if $\ch(F) \neq 2$ or if $J$ is simple (but not always \cite[1.31, 1.32]{MR634508}) --- there exists a commutative associative $F$-algebra $R$, necessarily unique, such that $F[x] = R^+$ \cite[Prop.~1]{MR0271181}, \cite[Prop.~4.6.2]{MR634508}. By abuse of language, we simply write $R = F[x]$ and say $R$ is a subalgebra of $J$.

In a slightly different, but similar, vein we wish to talk about \'etale subalgebras of a Jordan algebra. This is justified by the fact that \'etale $F$-algebras are comletely determined by their Jordan structure. More precisely, we have the following simple result.

\begin{lem} \label{l.PLUS} {Let $E,R$ be commutative associative $F$-algebras such that $E$ is finite-dimensional \'etale and $E^+ = R^+$ as Jordan algebras. Then $E = R$ as commutative associative algebras.}
\end{lem}

\begin{proof} Extending scalars if necessary, we may assume that $E$ as a (unital) $F$-algebra is generated by a single element $x \in E$. But since the powers of $x$ in $E$ agree with those in $E^+ = R^+$, hence with those in $R$, the assertion follows. \end{proof}

\subsection{Inverses and Jordan division algebras} \label{ss.INJODI} Let $J$ be a Jordan algebra over $F$. An element $x \in J$ is said to be \emph{invertible} if the $U$-operator $U_x\:J \to J$ is bijective (equivalently, $1_J \in \mathrm{Im}(U_x)$), in which case we call $x^{-1} := U_x^{-1}x$ the \emph{inverse} of $x$ in $J$. Invertibility and inverses are preserved by homomorphisms. It follows from the fundamental formula \eqref{FUFO} that, if $x,y \in J$ are invertible, then so is $U_xy$ and $(U_xy)^{-1} = U_{x^{-1}}y^{-1}$. Moreover, setting $x^n := (x^{-1})^{-n}$ for $n \in \IZ$, $n < 0$, we have \eqref{PORU} for all $m,n,p \in \IZ$. In agreement with earlier conventions, the set of invertible elements in $J$ will be denoted by $J^\times$. If $J^\times = J \setminus \{0\} \neq \emptyset$, then we call $J$ a \emph{Jordan division algebra}.
If $A$ is an associative algebra, then $(A^+)^\times = A^\times$, and the inverses are the same. Similarly, if $(B,\tau)$ is an associative algebra with involution, then $\Symd(B,\tau)^\times =
\Symd(B,\tau) \cap B^\times$, $H(B,\tau)^\times = H(B,\tau) \cap B^\times$, and, again, in both cases, the inverses are the same. 

\subsection{Isotopes} \label{ss.IS} Let $J$ be a Jordan algebra over $F$ and $p \in J^\times$. Then the vector space $J$ together with the $U$-operator $U_x^{(p)} := U_xU_p$ and the distinguished element $1_J^{(p)} := p^{-1}$ is a Jordan algebra over $F$, called the \emph{$p$-isotope} (or simply an \emph{isotope}) of $J$ and denoted by $J^{(p)}$. We have $J^{(p)\times} = J^\times$ and $(J^{(p)})^{(q)} = J^{(U_pq)}$ for all $q \in J^\times$, which implies $(J^{(p)})^{(q)} =J$ for $q := p^{-2}$. Passing to isotopes is functorial in the following sense: If $\varphi\:J \to J^\prime$ is a homomorphism of Jordan algebras, then so is $\varphi\:J^{(p)} \to J^{\prime(\varphi(p))}$, for any $p \in J^\times$.

Let $A$ be an associative algebra over $F$ and $p \in  (A^+)^\times = A^\times$.  Then right multiplication by $p$ in $A$ gives an isomorphism $R_p\:(A^+)^{(p)} \overset{\sim}\to A^+$ of Jordan algebras. On the other hand, if $(B,\tau)$ is an associative algebra with involution, then so is $(B,\tau^{(p)})$, for any $p \in H(B,\tau)^\times$, where $\tau^{(p)}\:B \to B$ via $x \mapsto p^{-1}\tau(x)p$ stands for the \emph{$p$-twist} of $\tau$, and 
\begin{align}
\label{RIP} R_p\:H(B,\tau)^{(p)} \overset{\sim} \longrightarrow H(B,\tau^{(p)})
\end{align} 
is an isomorphism of Jordan algebras.

\subsection{Homotopies  and the structure group} \label{ss.HOI} If $J,J^\prime$ are Jordan algebras over $F$, a \emph{homotopy} from $J$ to $J^\prime$ is a homomorphism $\varphi\:J \to J^{\prime(p^\prime)}$ of Jordan algebras, for some $p^\prime \in J^{\prime\times}$. In this case, $p^\prime = \varphi(1_J)^{-1}$ is uniquely determined by $\varphi$. Bijective homotopies are called \emph{isotopies}, while injective homotopies are called \emph{isotopic embeddings}. The set of isotopies from $J$ to itself is a subgroup of $\GL(J)$, called the \emph{structure group} of $J$ and denoted by $\Str(J)$.  It consists of all linear bijections $\eta\:J \to J$ such that some linear bijection $\eta^\sharp\:J \to J$ satisfies $U_{\eta(x)} = \eta U_x\eta^\sharp$ for all $x \in J$. The structure group contains the automorphism group of $J$ as a subgroup; more precisely, $\Aut(J)$ is the stabilizer of $1_J$ in $\Str(J)$. Finally, thanks to the fundamental formula \eqref{FUFO}, we have $U_y \in \Str(J)$ for all $y \in J^\times$. 

% % % % % % % % % % % % % % % % % % % % % % % % % % % % % % % % % % % % % % % % % % % % % % % % % % %

\section{Cubic Jordan algebras} \label{s.CUJO}  In this section, we recall the main ingredients of the approach to a particularly important class of Jordan algebras through the formalism of cubic norm structures. Our main references are \cite{MR0238916} and \cite{MR0404370}. Systematic use will be made of the following notation: given a polynomial map $P\:V \to W$ between vector spaces $V,W$ over $F$ and $y \in V$, we denote by $\partial_yP\:V \to W$ the polynomial map given by the derivative of $P$ in the direction $y$.

\subsection{Cubic norm structures} \label{ss.CONO} By a \emph{cubic norm structure} over $F$ we mean a quadruple $X = (V,c,\sharp,N)$ consisting of a vector space $V$ over $F$, a distinguished element $c \in V$ (the \emph{base point}), a quadratic map $x \mapsto x^\sharp$ from $V$ to $V$ (the \emph{adjoint}), with bilinearization $x \times y := (x + y)^\sharp - x^\sharp - y^\sharp$, and a cubic form $N\:V \to F$ (the \emph{norm}), such that, writing
\begin{align*}
T(y,z) := (\partial_yN)(c)(\partial_zN)(c) - (\partial_y\partial_zN)(c) &&(y,z \in V)
\end{align*}
for the (bilinear) \emph{trace} of $X$ and $T(y) := T(y,c)$ for the linear one, the equations
\begin{align}
\label{BAPO} c^\sharp = c,\;\;& N(c) = 1 &&(\text{base point identities}), \\
\label{UNID} c \times x =\,\,&T(x)c - x &&(\text{unit identity}), \\
\label{GRID} (\partial_yN)(x) =\,\,&T(x^\sharp,y) &&(\text{gradient identity}), \\
\label{ADID} x^{\sharp\sharp} =\,\,&N(x)x &&(\text{adjoint identity})
\end{align}
hold in all scalar extensions. A subspace of $V$ is called a \emph{cubic subnorm structure} of $X$ if it contains the base point and is stable under the adjoint map.; it may then canonically be regarded  as a cubic norm structure in its own right. A \emph{homomorphism} of cubic norm structures is a linear map of the underlying vector spaces preserving base points, adjoints and norms. A cubic norm structure $X$ as above is said to be \emph{non-singular} if $V$ has finite dimension over $F$ and the bilinear trace $T\:V \times V \to F$ is a non-degenerate symmetric  bilinear form. If $X$ and $Y$ are cubic norm structures over $F$, with $Y$ non-singular, and $\varphi\:X \to Y$ is a surjective linear map preserving base points and norms, then $\varphi$ is an isomorphism of cubic norm structures \cite[p.~507]{MR0238916}.

\subsection{The associated Jordan algebra} \label{ss.ASJO} Let $X = (V,c,\sharp,N)$ be a cubic norm structure over $F$ and write $T$ for its bilinear trace. Then the $U$-operator
\begin{align}
\label{UCUB} U_xy := T(x,y)x - x^\sharp \times y 
\end{align} 
and the base point $c$ convert the vector space $V$ into a Jordan algebra over $F$, denoted by $J(X)$ and called the Jordan algebra \emph{associated with} $X$. The construction of $J(X)$ is clearly functorial in $X$.  We have
\begin{align}
\label{NOU} N(U_xy) = N(x)^2N(y) &&(x,y \in J).
\end{align}
Jordan algebras isomorphic to $J(X)$ for some cubic norm structure $X$ over $F$ are said to be \emph{cubic}. For example, let $J$ be a Jordan algebra over $F$ that is generically algebraic (e.g., finite-dimensional) of degree $3$ over $F$. Then $X = (V,c,\sharp,N)$, where $V$ is the vector space underlying $J$, $c := 1_J$, $\sharp$ is the numerator of the inversion map, and $N := N_J$ is the generic norm of $J$, is a cubic norm structure over $F$ satisfying $J = J(X)$; in particular, $J$ is a cubic Jordan algebra. In view of this correspondence, we rarely distinguish carefully between a cubic norm structure and its associated Jordan algebra. Non-singular cubic Jordan algebras, i.e., Jordan algebras arising from non-singular cubic norm structures, by \cite[p.~507]{MR0238916} have no absolute zero divisors, so $U_x = 0$ implies $x = 0$. 

\subsection{Cubic \'etale algebras} \label{ss.CUET}  Let $E$ be a cubic \'etale $F$-algebra. Then Lemma~\ref{l.PLUS} allows us to identify $E = E^+$ as a generically algebraic Jordan algebra of degree $3$ (with $U$-operator $U_xy = x^2y$), so we may write $E = E^+ = J(V,c,\sharp,N)$ as in \ref{ss.ASJO}, where $c = 1_E$ is the unit element, $\sharp$ is the adjoint and $N = N_E$ is the norm of $E = E^+$. We also write $T_E$ for the (bilinear) trace of $E$. The discriminant (algebra) of $E$ will be denoted by $\Delta(E)$; it is a quadratic \'etale $F$-algebra \cite[18.C]{KMRT}.

\subsection{Isotopes of cubic norm structures} \label{ss.ISCU} Let $X = (V,c,\sharp,N)$ be a cubic norm structure over $F$. An element $p \in V$ is invertible in $J(X)$ if and only if $N(p) \neq 0$, in which case $p^{-1} = N(p)^{-1}p^\sharp$. Moreover,
\[
X^{(p)} := (V,c^{(p)},\sharp^{(p)},N^{(p)}),
\]
with $c^{(p)} := p^{-1}$, $x^{\sharp^{(p)}} := N(p)U_p^{-1}x^\sharp$, $N^{(p)} := N(p)N$, is again a cubic norm structure over $F$ , called the \emph{$p$-isotope} of $X$. This terminology is justified since the associated Jordan algebra $J(X^{(p)}) = J(X)^{(p)}$ is the $p$-isotope of $J(X)$. We also note that the bilinear trace of $X^{(p)}$ is given by
\begin{align}
\label{TRIS} T^{(p)}(y,z) = T(U_py,z) &&(y,z \in X)
\end{align}
in terms of the bilinear trace $T$ of $X$.  Combining the preceding considerations with \ref{ss.CONO}, we conclude that the structure group of a \emph{non-singular} cubic Jordan algebra agrees with its group of norm similarities.

\subsection{Cubic Jordan matrix algebras} \label{ss.CUMA} Let $C$ be a composition algebra over $F$, with norm $n_C$, trace $t_C$, and conjugation $v \mapsto \bar v := t_C(v)1_C - v$. For any diagonal matrix
\[
\Gamma = \diag(\gamma_1,\gamma_2,\gamma_3) \in \GL_3(F),
\]
the pair
\begin{align*}
\big(\Mat_3(C),\tau_\Gamma\big), \quad \tau_\Gamma(x) := \Gamma^{-1}\bar x^t\Gamma  &&(x \in \Mat_3(C)),
\end{align*}
is a non-associative $F$-algebra with involution, allowing us to consider the subspace
\[
\Her_3(C,\Gamma) := \Symd\big(\Mat_3(C),\tau_\Gamma\big) = \{x + \tau_\Gamma(x) \mid x \in \Mat_3(C)\} \subseteq \Mat_3(C),
\]
which is easily seen to agree with the space of elements $x \in \Mat_3(C)$ that are $\Gamma$-hermitian ($x = \Gamma^{-1}\bar x^t\Gamma$) and have scalars down the diagonal (the latter condition being automatic for $\ch(F) \neq 2$). In terms of the usual matrix units $e_{ij} \in \Mat_3(C)$, $1 \leq i,j \leq 3$, we therefore have
\[
\Her_3(C,\Gamma) = \sum (Fe_{ii} + C[jl]),
\]
the sum on the right being taken over all cyclic permutations $(ijl)$ of $(123)$, where
\[
C[jl] := \{v[jl] \mid v \in C\}, \quad v[jl] := \gamma_lve_{jl} + \gamma_j\bar ve_{lj}. 
\]
Now put $V := \Her_3(C,\Gamma)$ as a vector space over $F$, $c := \Eins_3$ (the $3 \times 3$ unit matrix) and define adjoint and norm on $V$ by
\begin{align*}
x^\sharp :=\,\,&\sum \Big(\big(\alpha_j\alpha_l - \gamma_j\gamma_ln_C(v_i)\big)e_{ii} + \big(-\alpha_iv_i + \gamma_i\overline{v_jv_l}\big)[jl]\Big), \\
N(x) :=\,\,&\alpha_1\alpha_2\alpha_3 - \sum \gamma_j\gamma_l\alpha_in_C(v_i) + \gamma_1\gamma_2\gamma_3t_C(v_1v_2v_3)
\end{align*}
for all $x = \sum(\alpha_ie_{ii} + v_i[jl])$ in all scalar extensions of $V$. Then $X := (V,c,\sharp,N)$ is a cubic norm structure over $F$. Henceforth, the symbol $\Her_3(C,\Gamma)$ will stand for this cubic norm structure but also for its associated cubic Jordan algebra. We always abbreviate $\Her_3(C) := \Her_3(C,\Eins_3)$.

\subsection{Albert algebras} \label{ss.ALAL} Writing $\Zor(F)$ for the split octonion algebra of Zorn vector matrices over $F$, the cubic Jordan matrix algebra $\Her_3(\Zor(F))$ is called the \emph{split Albert algebra} over $F$. By an \emph{Albert algebra} over $F$, we mean an $F$-form of $\Her_3(\Zor(F))$, i.e., a Jordan algebra over $F$ (necessarily absolutely simple and non-singular of degree $3$ and dimension $27$) that becomes isomorphic to the split Albert algebra when extending scalars to the separable closure. Albert algebras are either \emph{reduced}, hence have the form $\Her_3(C,\Gamma)$ as in \ref{ss.CUMA}, $C$ an octonion algebra over $F$ (necessarily unique), or are cubic Jordan division algebras.

\subsection{Associative algebras of degree $3$ with unitary involution} \label{ss.ASIN} By an \emph{associative algebra of degree $3$ with unitary involution} over $F$ we mean a triple $(K,B,\tau)$ with the following properties: $K$ is a quadratic \'etale $F$-algebra, with norm $n_K$, trace $t_K$ and conjugation $\iota_K$, $a \mapsto \bar a$, $B$ is an associative algebra of degree $3$ over $K$ and $\tau\:B \to B$ is an $F$-linear involution that induces the conjugation of $K$ via restriction. All this makes obvious sense even in the special case that $K \cong F \times F$ is split, as do the generic norm, trace and adjoint of $B$, which are written as $N_B,T_B,\sharp$, respectively, connect naturally with the involution $\tau$ and agree with the corresponding notions for the cubic Jordan algebra $B^+$. In particular, $H(B,\tau)$ is a Jordan algebra of degree $3$ over $F$ whose associated cubic norm structure derives from that of $B^+$ via restriction.

An associative algebra $(K,B,\tau)$ of degree $3$ with unitary involution over $F$ is said to be \emph{non-singular} if the corresponding cubic Jordan algebra $B^+$ has this property, equivalently, if $B$ is free of finite rank over $K$  and $T_B\:B \times B \to K$ is a non-degenerate symmetric bilinear form in the usual sense.
 
\subsection{The second Tits construction} \label{ss.SETI} Let $(K,B,\tau)$ be an associative algebra of degree $3$ with unitary involution over $F$ and suppose we are given invertible elements $u \in H(B,\tau)$, $\mu \in K$ such that $N_B(u) = n_K(\mu)$. We put $V := H(B,\tau) \oplus Bj$ as the external direct sum of $H(B,\tau)$ and $B$ as vector spaces over $F$ to define base point, adjoint and norm on $V$ by the formulas
\begin{align}
\label{UNSE} c :=\,\,&1_B + 0\cdot j, \\
\label{ADSE} x^\sharp :=\,\,&(v_0^\sharp - vu\bar v) + (\bar\mu\bar v^\sharp u^{-1} - v_0v)j, \\
\label{NOSE} N(x) :=\,\,&N_B(v_0) + \mu N_B(v) + \bar\mu \overline{N_B(v)} - T_B\big(v_0,vu\tau(v)\big) 
\end{align}
for $x = v_0 + vj$, $v_0 \in H(B,\tau)$, $v \in B$ (and in all scalar extensions as well). Then we obtain a cubic norm structure $X := (V,c,\sharp,N)$ over $F$ whose associated cubic Jordan algebra will be denoted by $J := J(K,B,\tau,u,\mu) := J(X)$ and has the bilinear trace
\begin{align}
\label{BILSE} T(x,y) =\,\,&T_B(v_0,w_0) + T_B\big(vu\tau(w)\big) + T_B\big(wu\tau(v)\big) \notag \\
=\,\,&T_B(v_0,w_0) + t_K\Big(T_B\big(vu\tau(w)\big)\Big)
\end{align} 
for $x$ as above and $y = w_0 + wj$, $w_0 \in H(B,\tau)$, $w \in B$. It follows that, if $(K,B,\tau)$ is non-singular, then so is $J$. Note also that the cubic Jordan algebra $H(B,\tau)$ identifies with a subalgebra of $J$ through the initial summand. 
 
If, in addition to the above, $(B,\tau)$ is central simple as an algebra with involution over $F$, then $K$ is the centre of $B$, $J(B,\tau,u,\mu) := J(K,B,\tau,u,\mu)$ is an Albert algebra over $F$, and all Albert algebras can be obtained in this way. More precisely, every Albert algebra $J$ over $F$ contains a subalgebra isomorphic to $H(B,\tau)$ for some central simple associative algebra $(B,\tau)$ of degree $3$ with unitary involution over $F$, and every homomorphism $H(B,\tau) \to J$ can be extended to an isomorphism from $J(B,\tau,u,\mu)$ to $J$, for some invertible elements $u \in H(B,\tau)$, $\mu \in K$ satisfying $N_B(u) = n_K(\mu)$. 

 Our next result is a variant of \cite[Prop.~3.9]{MR86g:17020} which extends the isomorphism \eqref{RIP} in a natural way.

\begin{lem} \label{l.SETIS} {Let $(K,B,\tau)$ be a non-singular associative algebra of degree $3$ with unitary involution over $F$ and suppose $u \in H(B,\tau)$, $\mu \in K$ are invertible elements satisfying $N_B(u) = n_K(\mu)$. Then, given any $p \in H(B,\tau)^\times$, writing $\tau^{(p)}$ for the $p$-twist of $\tau$ in the sense of} \ref{ss.IS} \emph{and setting $u^{(p)} := p^\sharp u$, $\mu^{(p)} := N_B(p)\mu$, the following statements hold.} 
\begin{enumerate}
\renewcommand{\theenumi}{\alph{enumi}}
\item {$u^{(p)} \in H(B,\tau^{(p)})^\times$, $N_B(u^{(p)}) = n_K(\mu^{(p)})$ and $H(B,\tau^{(p)}) = H(B,\tau)p$.} 
\item {The map
\[
\hat{R}_p\:J(K,B,\tau,u,\mu)^{(p)} \overset{\sim} \longrightarrow J(K,B,\tau^{(p)},u^{(p)},\mu^{(p)}), \quad v_0 + vj \longmapsto v_0p + (p^{-1}vp)j,
\]
is an isomorphism of cubic Jordan algebras.}
\end{enumerate}
\end{lem}

\begin{proof} (a) From $p^{-1} = N_B(p)^{-1}p^\sharp$ we conclude $pp^\sharp = N_B(p)1_B = p^\sharp p$, which implies the first assertion, but also the second since $N_B(p^\sharp) = N_B(p)^2$. The third one follows from \eqref{RIP}.

(b) By (a), \eqref{TRIS} and \ref{ss.SETI}, the map $\hat{R}_p$ is a linear bijection between non-singular cubic Jordan algebras preserving base points. By \ref{ss.CONO}, it therefore suffices to show that it preserves norms as well. Writing $N$ (resp.~$N^\prime$) for the norm of $J(K,B,\tau,u,\mu)$ (resp.~$J(K,B,\tau^{(p)}, u^{(p)},\mu^{(p)})$, we let $v_0 \in H(B,\tau)$, $v \in B$ and compute, using \eqref{NOSE},
\begin{align*}
(N^\prime \circ \hat{R}_p)(v_0 + vj) =\,\,&N^\prime(v_0p + (p^{-1}vp)j) \\
=\,\,&N_B(p)N_B(v_0) + N_B(p)\mu N_B(v) + N_B(p)\bar\mu\overline{N_B(v)} \\
\,\,&- T_B\big(v_0pp^{-1}vpp^\sharp u\tau^{(p)}(p^{-1}vp)\big) \\
=\,\,&N_B(p)\Big(N_B(v_0) + \mu N_B(v) + \bar{\mu}\overline{N_B(v)} - T_B\big(v_0vu\tau(v)\big)\Big) \\
=\,\,&N^{(p)}(v_0 + vj), 
\end{align*}
as desired. \end{proof}

\begin{rmk} The lemma holds without the non-singularity condition on $(K,B,\tau)$ but the proof is more involved.  \end{rmk}

 If the quadratic \'etale $F$-algebra $K$ in \ref{ss.SETI} is split, there is a less cumbersome way of describing the output of the second Tits construction.  
 
\subsection{The first Tits construction} \label{ss.FITI} Let $A$ be an associative algebra of degree $3$ over $F$ and $\mu \in F^\times$. Put $V := A \oplus Aj_1 \oplus Aj_2$ as the direct sum of three copies of $A$ as an $F$-vector space and define base point, adjoint and norm on $V$ by the formulas $c := 1_A + 0\cdot j_1 + 0\cdot j_2$,
\begin{align}
\label{SHAFI} x^\sharp :=\,\,&(v_0^\sharp - v_1v_2) + (\mu^{-1}v_2^\sharp - v_0v_1)j_1 + (\mu v_1^\sharp - v_2v_0)j_2, \\
\label{NOFI} N(x) :=\,\,&N_A(v_0) + \mu N_A(v_1) + \mu^{-1}N_A(v_2) - T_A(v_0v_1v_2)
\end{align}
for $x = v_0 + v_1j_1 + v_2j_2$, $v_0,v_1,v_2$ running over all scalar extensions of $A$. Then $X := (V,c,\sharp,N)$ is a cubic norm structure over $F$, with bilinear trace given by
\begin{align}
\label{TRAFI} T(x,y) = T_A(v_0,w_0) + T_A(v_1,w_2) + T_A(v_2,w_1)
\end{align}
for $x$ as above and $y = w_0 + w_1j_1 + w_2j_2$, $w_0,w_1,w_2 \in A$. The associated cubic Jordan algebra will be denoted by $J(A,\mu) := J(X)$. The Jordan algebra $A^+$ identifies with a cubic subalgebra of $J(A,\mu)$ through the initial summand, and if $A$ is central simple, then $J(A,\mu)$ is an Albert algebra, which is either split or division.
 
Now let $(K,B,\tau)$ be an associative algebra of degree $3$ with unitary involution over $F$ and suppose $\mu \in K^\times$, $u \in H(B,\tau)^\times$ satisfy $n_K(\mu) = N_B(u)$. If $K = F \times F$ is split, then $(B,\tau)$ identifies with $(A \times A^{\mathrm{op}},\varepsilon)$ for some associative algebra $A$ of degree $3$ over $F$, where $\varepsilon$ denotes the exchange involution. Moreover, $\mu = (\alpha,\beta)$, where $\alpha \in F$ is invertible, $\beta = \alpha^{-1}N_B(u)$, and there exists a canonical isomorphism $J := J(K,B,\tau,u,\mu) \cong J(A,\alpha) =:J^\prime$ matching $H(A \times A^{\mathrm{op}},\varepsilon)$ canonically with $A^+$ as subalgebras of $J,J^\prime$, respectively. On the other hand, if $K$ is a field, the preceding considerations apply to the base change from $F$ to $K$ and then yield an isomorphism $J(K,B,\tau,u,\mu)_K \cong J(B,\mu)$.
 
 % % % % % % % % % % % % % % % % % % % % % % % % % % % % % % % % % % % % % % % % % % % % % % % % % % %
 
\section{The weak and strong Skolem-Noether properties} \label{s.DEFI}  As we have pointed out  in \ref{ss.SKONOAL}, extending an isomorphism between cubic \'etale subalgebras of an Albert algebra $J$ to an automorphism on all of $J$ will in general not be possible. Working with elements of the structure group rather than automorphisms, our Theorem \ref{t.SNOTWI} above is supposed to serve as a substitute for this deficiency. Unfortunately, however, this substitute suffers from deficiencies of its own since the natural habitat of the structure group is the category of Jordan algebras \emph{not} under homomorphisms \emph{but, instead,} under homotopies. 

Fixing a cubic Jordan algebra $J$ over our base field $F$ and a cubic \'etale $F$-algebra $E$ throughout this section, we therefore feel justified in phrasing the following formal definition.

\subsection{Weak and strong equivalence of isotopic embeddings} \label{ss.WESO} (a) Two isotopic embeddings $i,i^\prime\:E \to J$ in the sense of \ref{ss.HOI} are said to be \emph{weakly equivalent} if there exist an element $w \in E$ of norm $1$ and an element $\varphi \in \Str(J)$ such that the diagram
\begin{align}
\vcenter{\label{ISTEM} \xymatrix@R=0.5pc{
E \ar[rr]_{R_w} \ar[dd]_{i^\prime} &&
E \ar[dd]^{i} \\ \\
J \ar[rr]_{\varphi} && J}}
\end{align}
commutes. They are said to be \emph{strongly equivalent} if $\varphi \in \Str(J)$ can furthermore be chosen so that the diagram commutes with $w = 1$ (i.e., $R_w = \Id_E$).  
Weak and strong equivalence clearly define equivalence relations on the set of isotopic embeddings from $E$ to $J$. 

(b) The pair $(E,J)$ is said to satisfy the \emph{weak (resp.~strong) Skolem-Noether property for isotopic embeddings} if any two isotopic embeddings from $E$ to $J$ are weakly (resp.~strongly) equivalent.  The weak (resp.~strong) Skolem-Noether property for isomorphic embeddings is defined similarly, by restricting the maps $i$, $i'$ to be isomorphic embeddings instead of merely isotopic  ones.

\begin{rmk} \label{ss.SKALG}  In \ref{ss.WESO} we have defined four different properties, depending on whether one considers the weak or strong Skolem-Noether property for isotopic or isomorphic embeddings.  Clearly the combination weak/isomorphic is the weakest of these four properties and strong/isotopic is the strongest.

In the case where $J$ is an Albert algebra, Theorem \ref{t.SNOTWI} is equivalent to saying that the pair $(E,J)$ satisfies the weakest combination, the weak Skolem-Noether property for isomorphic embeddings.  On the other hand, suppose $i,i^\prime\:E \to J$ are isomorphic embeddings and $\varphi \in \Str(J)$ makes \eqref{ISTEM} commutative with $w = 1$. Then $\varphi$ fixes $1_J$ and hence is an automorphism of $J$. But such an automorphism will in general not exist \cite[Th.~9]{MR0088487}, and if it doesn't the pair $(E,J)$ will fail to satisfy the strong Skolem-Noether property for isomorphic embeddings.  In view of this failure, we are led quite naturally to the following (as yet) open question:
\begin{equation} \label{ss.OPQUE}
\parbox{4in}{\emph{Does the pair $(E,J)$, with $J$ absolutely simple (of degree $3$), always satisfy the weak Skolem-Noether property for isotopic embeddings?}}
\end{equation}
This is equivalent to asking whether, given two cubic \'etale subalgebras $E_1 \subseteq J^{(p_1)}$, $E_2 \subseteq J^{(p_2)}$ for some $p_1,p_2 \in J^\times$, every isotopy $\eta\:E_1 \to E_2$ allows a norm-one element $w \in E_1$ such that the isotopy $\eta \circ R_w\:E_1 \to E_2$ extends to an element of the structure group of $J$.
Regrettably, the methodological arsenal assembled in the present paper, consisting as it does of rather elementary manipulations involving the two Tits constructions, does not seem strong enough to provide an affirmative answer to this question. 

But in the case where $J$ is absolutely simple of dimension 9 --- i.e., the Jordan algebra of symmetric elements in a central simple associative algebra of degree $3$ with unitary involution over $F$ \cite[15.5]{MR946263} --- we will show in Th.~\ref{t.SKONOIT} below that the weak Skolem-Noether property for isotopic embeddings does hold.  This result, in turn, will be instrumental in proving Theorem \ref{t.SNOTWI} in section \S\ref{s.ALAL}.
\end{rmk}
 
 In phrasing Open Question~\ref{ss.OPQUE}, we could have gone one step further by bringing the theory of Jordan pairs \cite{MR0444721} into play. We will not do so since our methods do not readily adapt to the Jordan pair setting. Instead, we will confine ourselves to making the following remark.
 
\begin{rmk} \label{r.JOFRA} Assume in \ref{ss.OPQUE} that $E \cong F \times F \times F$ is split. Giving an isotopic embedding from $E$ to $J$ is then equivalent to giving a frame, necessarily of length $3$, in the Jordan pair $\mfV:= (J,J)$. But following Loos \cite[Cor.~3 of Th.~2]{MR1100842}, the diagonal Peirce components of two ordered frames in $\mfV$ can be matched by some element in the elementary group of $\mfV$, i.e., in a certain subgroup of the structure group of $J$, and this fact is easily seen to translate into the commutative diagram \eqref{ISTEM} after an appropriate choice of $w \in E^\times$ (possibly not of norm $1$) and $\varphi \in \Str(J)$.
\end{rmk}
 
 % % % % % % % % % % % % % % % % % % % % % % % % % % % % % % % % % % % % % % % % % % % % % % % % % % %
 
\section{Cubic Jordan algebras of dimension $9$} \label{s.CUJONI}  Our goal in this section will be to answer Question~\ref{ss.OPQUE} affirmatively in case $J$ is a nine-dimensional absolutely simple cubic Jordan algebra over $F$. Before we will be  able to do so, a few preparations are required.
 
\subsection{Quadratic and cubic \'etale algebras} \label{ss.QUACU} (a) If $K$ and $L$ are quadratic \'etale algebras over $F$, then so is
\[
K \ast L := H(K \otimes L,\iota_K \otimes \iota_L),
\]
where $\iota_K$ and $\iota_L$ denote the conjugations of $K$ and $L$, respectively. The composition $(K,L) \mapsto K \ast L$ corresponds to the abelian group structure of $H^1(F,\IZ/2\IZ)$, which classifies quadratic \'etale $F$-algebras. 

(b) Next suppose $L$ and $E$ are a quadratic and cubic \'etale $F$-algebras, respectively.  Then $E \otimes L$ may canonically be viewed as a cubic \'etale $L$-algebra, whose norm, trace, adjoint will again be denoted by $N_E,T_E,\sharp$, respectively. On the other hand, $E \otimes L$ may also be viewed canonically as a quadratic \'etale $E$-algebra, whose norm, trace and conjugation will again be denoted by $n_L$, $t_L$, and $\iota_L$, $x \mapsto \bar x$, respectively. We may and always will identify $E \subseteq E \otimes L$ through the first factor and then have $E = H(E \otimes L,\iota_L)$.

\subsection{The \'etale Tits process} \cite[1.3]{MR2032452} \label{ss.DETI} Let $L$, resp.~$E$, be a quadratic, resp cubic, \'etale algebra over $F$ and as in \ref{ss.CUET} write $\Delta(E)$ for the discriminant of $E$, which is a quadratic \'etale $F$-algebra. With the conventions of \ref{ss.QUACU}~(b), the triple $(K,B,\tau) := (L,E \otimes L,\iota_L)$ is an associative algebra of degree $3$ with unitary involution over $F$ in the sense of \ref{ss.ASIN} such that $H(B,\tau) = E$. Hence, if $u \in E$ and $b \in L$ are invertible elements satisfying $N_E(u) = n_L(b)$, the second Tits construction \ref{ss.SETI} leads to a cubic Jordan algebra
\[
J(E,L,u,b) := J(K,B,\tau,u,b) = J(L,E \otimes L,\iota_L,u,b)
\]
that belongs to the cubic norm structure $(V,c,\sharp,N)$ where $V = E \oplus (E \otimes L)j$ as a vector space over $F$ and $c,\sharp,N$ are defined by \eqref{UNSE}--\eqref{NOSE} in all scalar extensions. The cubic Jordan algebra $J(E,L,u,b)$ is said to arise from $E,L,u,b$ by means of the \emph{\'etale Tits process}. There exists a central simple associative algebra $(B,\tau)$ of degree $3$ with unitary involution over $F$ uniquely determined by the condition that $J(E,L,u,b) \cong H(B,\tau)$, and by \cite[Th.~1]{MR86g:17020}, the centre of $B$ is isomorphic to $\Delta(E) \ast L$ (cf. \ref{ss.QUACU}~(a)) as a quadratic \'etale $F$-algebra.

For convenience, we now recall three results from \cite{MR2032452} that will play a crucial role in providing an affirmative answer to Question~\ref{ss.OPQUE} under the conditions spelled out at the beginning of this section.

\begin{thm} \label{t.IMEX} \emph{(\cite[1.6]{MR2032452})} {Let $E$ be a cubic \'etale $F$-algebra, $(B,\tau)$ a central simple associative algebra of degree $3$ with unitary involution over $F$ and suppose $i$ is an isomorphic embedding from $E$ to $H(B,\tau)$. Writing $K$ for the centre of $B$ and setting $L := K \ast \Delta(E)$, there are invertible elements $u \in E$, $b \in L$ satisfying $N_E(u) = n_L(b)$ such that $i$ extends to an isomorphism from the \'etale Tits process algebra $J(E,L,u,b)$ onto $H(B,\tau)$.} \hfill $\qed$
\end{thm}

\begin{thm} \label{t.IMCRI} \emph{(\cite[3.2]{MR2032452})} {Let $E,E^\prime$ and $L,L^\prime$ be cubic and quadratic \'etale algebras, respectively, over $F$ and suppose we are given invertible elements $u \in E$, $u^\prime \in E^\prime$, $b \in L$, $b^\prime \in L^\prime$ satisfying $N_E(u) = n_L(b)$, $N_{E^\prime}(u^\prime) = n_{L^\prime}(b^\prime)$. We write 
\[
J := J(E,L,u,b) = E \oplus (E \otimes L)j, \quad J^\prime := J(E^\prime,L^\prime, u^\prime,b^\prime) = E^\prime \oplus (E^\prime \otimes L^\prime)j^\prime 
\]
as in} \ref{ss.DETI} {for the corresponding \'etale Tits process algebras and let $\varphi\:E^\prime \overset{\sim}\to E$ be an isomorphism. Then, for an arbitrary map $\Phi\:J^\prime \to J$, the following conditions are equivalent.}
\begin{itemize}
\item [(i)] {$\Phi$ is an isomorphism extending $\varphi$.} 

\item [(ii)] {There exist an isomorphism $\psi\:L^\prime \overset{\sim}\to L$ and an invertible element $y \in E \otimes L$ such that $\varphi(u^\prime) = n_L(y)u$, $\psi(b^\prime) = N_E(y)b$ and
\begin{align}
 \Phi(v_0^\prime + v^\prime j^\prime) = \varphi(v_0^\prime) + \big(y(\varphi \otimes \psi)(v^\prime)\big)j
\end{align}
for all $v_0^\prime \in E^\prime$, $v^\prime \in E^\prime \otimes L^\prime$.} \hfill $\qed$
\end{itemize}
\end{thm}

\begin{prop} \emph{(\cite[4.3]{MR2032452})} \label{p.ETFIM} {Let $E$ be a cubic \'etale $F$-algebra and $\alpha,\alpha^\prime \in F^\times$. Then the following conditions are equivalent.}
\begin{enumerate}
\item  {The first Tits constructions $J(E,\alpha)$ and $J(E,\alpha^\prime)$} (\emph{cf.}~\ref{ss.FITI}) \emph{are isomorphic.}

\item {$J(E,\alpha)$ and $J(E,\alpha^\prime)$ are isotopic.}

\item  {$\alpha \equiv \alpha^{\prime\varepsilon} \bmod N_E(E^\times)$ for some $\varepsilon = \pm 1$.} 

\item {The identity of $E$ can be extended to an isomorphism from $J(E,\alpha)$ onto $J(E,\alpha^\prime)$.} \hfill $\qed$
\end{enumerate}
\end{prop}

 Our next aim will be to derive a version of Th.~\ref{t.IMEX} that works with isotopic rather than isomorphic embeddings and brings in a normalization condition already known from \cite[(39.2)]{KMRT}.

\begin{prop} \label{p.ITEX} {Let $(B,\tau)$ be a central simple associative algebra of degree $3$ with unitary involution over $F$ and write $K$ for the centre of $B$. Suppose further that $E$ is a cubic \'etale $F$-algebra and put $L := K \ast \Delta(E)$. Given any isotopic embedding $i\:E \to J := H(B,\tau)$, there exist elements $u \in E$, $b \in L$ such that $N_E(u) = n_L(b) = 1$ and $i$ can be extended to an isotopy from $J(E,L,u,b)$ onto $J$.} 
\end{prop}

\begin{proof} By \ref{ss.HOI}, some invertible element $p \in J$ makes $i\:E \to J^{(p)}$ an isomorphic embedding. On the other hand, invoking \ref{ss.IS} and writing $\tau^{(p)}$ for the $p$-twist of $\tau$, it follows that 
\[
R_p\:J^{(p)} \overset{\sim} \longrightarrow H(B,\tau^{(p)}) 
\]
is an isomorphism of cubic Jordan algebras, forcing $i_1 := R_p \circ i\:E \to H(B,\tau^{(p)})$ to be an isomorphic embedding. Hence Th.~\ref{t.IMEX} yields invertible elements $u_1 \in E$, $b_1 \in L$ such that $N_E(u_1) = n_L(b_1)$ and, adapting the notation of \ref{ss.SETI} to the present set-up in an obvious manner, $i_1$ extends to an isomorphism
\[
\eta_1^\prime\:J(E,L,u_1,b_1) = E \oplus (E \otimes L)j_1 \overset{\sim} \longrightarrow H(B,\tau^{(p)}).
\]
Thus $\eta_1 := R_{p^{-1}} \circ \eta_1^\prime\:J(E,L.u_1,b_1) \overset{\sim} \to J^{(p)}$ is an isomorphism, which may therefore be viewed as an isotopy from $J(E,L,u_1,b_1)$ onto $J$ extending $i$. Now put $u := N_E(u_1)^{-1}u_1^3$, $b := \bar b_1b_1^{-1}$ and $y := u_1 \otimes b_1^{-1} \in (E \otimes L)^\times$ to conclude $N_E(u) = n_L(b) = 1$ as well as $n_L(y)u_1 = u$, $N_E(y)b_1 = b$. Applying Th.~\ref{t.IMCRI} to $\varphi := \Eins_E$, $\psi := \Eins_L$, we therefore obtain an isomorphism
\[
\Phi\:J(E,L,u,b) \overset{\sim} \longrightarrow J(E,L,u_1,b_1), \quad v_0 +vj_1 \longmapsto v_0 + (yv)j
\]
of cubic Jordan algebras, and $\eta := \eta_1 \circ \Phi\:J(E,L,u,b) \to J$ is an isotopy of the desired kind. \end{proof} 

\begin{lem} \label{l.UNET} {Let $L$, resp.~$E$ be a quadratic, resp.~cubic \'etale algebra over $F$ and suppose we are given elements $u \in E$, $b \in L$ satisfying $N_E(u) = n_L(b) = 1$. Then $w := u^{-1} \in E$ has norm $1$ and $R_w\:E \to E$ extends to an isomorphism
\[
\hat{R}_w\:J(E,L,1_E,b) \overset{\sim} \longrightarrow J(E,L,u,b)^{(u)}, \quad v + xj \longmapsto (vw) + xj
\]
of cubic Jordan algebras.}
\end{lem}

\begin{proof} This follows immediately from Lemma~\ref{l.SETIS} for $(K,B,\tau) := (L,E \otimes L,\iota_L)$, $\mu := b$ and $p := u$. \end{proof} 

 We are now ready for the main result of this section. 

\begin{thm} \label{t.SKONOIT} {Let $(B,\tau)$ be a central simple associative algebra of degree $3$ with unitary involution over $F$ and $E$ a cubic \'etale $F$-algebra. Then the pair $(E,J)$ satifies the weak Skolem-Noether property for isotopic embeddings in the sense of} \ref{ss.WESO}~(b). 
\end{thm}

\begin{proof} Given two isotopic embeddings $i,i^\prime\:E \to J$, we must show that they are weakly equivalent. In order to do so, we write $K$ for the centre of $B$ as a quadratic \'etale $F$-algebra and put $L := K \ast \Delta(E)$. Then Prop.~\ref{p.ITEX} yields elements $u,u^\prime \in E$, $b,b^\prime \in L$ satisfying 
\begin{align}
\label{NOONE} N_E(u) = N_E(u^\prime) = n_L(b) = n_L(b^\prime) = 1
\end{align}
such that the isotopic embeddings $i,i^\prime$ can be extended to isotopies
\begin{align}
\label{EXIT} \eta\:J(E,L,u,b) = E \oplus (E \otimes L)j \longrightarrow J, \; \eta^\prime\:J(E,L,u^\prime, b^\prime) = E \oplus (E \otimes L)j^\prime \longrightarrow J,
\end{align}
respectively. We now distinguish the following two cases. 

\emph{\underline{Case~1}: $L \cong F \times F$ is split.} As we have noted in \ref{ss.FITI}, there exist elements $\alpha,\alpha^\prime \in F^\times$ and isomorphisms
\[
\Phi\:J(E,L,u,b) \overset{\sim} \longrightarrow J(E,\alpha), \quad \Phi^\prime\:J(E,L,u^\prime,b^\prime) \overset{\sim} \longrightarrow J(E,\alpha^\prime)
\]
extending the identity of $E$. Thus \eqref{EXIT} implies that $\Phi \circ \eta^{-1} \circ \eta^\prime \circ \Phi^{\prime -1}\:J(E,\alpha^\prime) \to J(E,\alpha)$ is an isotopy, and applying Prop.~\ref{p.ETFIM}, we find an isomorphism $\theta\:J(E,\alpha^\prime) \overset{\sim} \to  J(E,\alpha)$ extending the identity of $E$. But then $\varphi := \eta \circ \Phi^{-1} \circ \theta \circ \Phi^\prime \circ \eta^{\prime -1}\:J \longrightarrow J$ is an isotopy, hence belongs to the structure group of $J$, and satisfies 
\[
\varphi \circ i^\prime = \eta \circ \Phi^{-1} \circ \theta \circ \Phi^\prime \circ \eta^{\prime -1} \circ \eta^\prime\vert_E = \eta\vert_E = i.  
\]
Thus $i$ and $i^\prime$ are even strongly equivalent.  

\emph{\underline{Case 2}: $L$ is a field.} Since $J(E,L,u,b)$ and $J(E,L,u^\prime, b^\prime)$ are isotopic (via $\eta^{\prime -1} \circ \eta$), so are their scalar extensions from $F$ to $L$. From this and \ref{ss.FITI} we therefore conclude that $J(E \otimes L,b)$ and $J(E \otimes L,b^\prime)$ are isotopic over $L$. Hence, by Prop.~\ref{p.ETFIM},
\begin{align}
\label{BNOPM} b = b^{\prime\varepsilon}N_E(z)
\end{align} 
for some $\varepsilon = \pm 1$ and some $z \in (E \otimes L)^\times$. Now put $\varphi := \Eins_E$, $\psi := \iota_L$ and $y := u^\prime \otimes 1_L \in (E \otimes L)^\times$. Making use of \eqref{NOONE} we deduce $n_L(y)u^{\prime -1} = u^\prime$, $N_E(y)b^{\prime-1} = \bar b^\prime$. Hence Th.~\ref{t.IMCRI} shows that the identity of $E$ can be extended to an isomorphism
\[
\theta\:J(E,L,u^\prime,b^\prime) \overset{\sim} \longrightarrow J(E,L,u^{\prime -1},b^{\prime -1}),
\]  
and we still have $N_E(u^{\prime -1}) = n_L(b^{\prime -1}) = 1$. Thus, replacing $\eta^\prime$ by $\eta^\prime  \circ \theta^{-1}$ if necessary, we may assume $\varepsilon = 1$ in \eqref{BNOPM}, i.e.,
\begin{align}
\label{BNOP} b = b^\prime N_E(z).
\end{align}
Next put $\varphi := \Eins_E$, $\psi := \Eins_L$ and $y := z \in (E \otimes L)^\times$, $u_1 := n_L(y)u^\prime$, $b_1 := N_E(y)b^\prime = b$ (by \eqref{BNOP}). Taking $L$-norms in \eqref{BNOP} and observing \eqref{NOONE}, we conclude $N_E(y)\overline{N_E(y)} = n_L\big(N_E(z)) = 1$, and since $u_1 = y \bar y u^\prime$, this implies $N_E(u_1) = 1$. Hence Th.~\ref{t.IMCRI} yields an isomorphism
\[
\theta\:J(E,L,u_1,b_1) \overset{\sim} \longrightarrow J(E,L,u^\prime,b^\prime)
\]
extending the identity of $E$, and replacing $\eta^\prime$ by $\eta^\prime \circ \theta$ if necessary, we may and from now on will assume
\begin{align}
\label{BNO} b = b^\prime.
\end{align}
Setting $w := u^{-1}$ and consulting Lemma~\ref{l.UNET}, we have $N_E(w) = 1$ and obtain a commutative diagram
\[
\xymatrix@R=0.5pc{
E \ar[rr]_{R_w} \,\,\ar@{^{(}->}[dd] && E \ar[rr]_{i} \,\,\ar@{^{(}->}[dd] && J \\ \\
J(E,L,1_E,b) \ar[rr]_{\hat{R}_w} && J(E,L,u,b) \ar[rruu]_{\eta},}
\]
where $\eta \circ \hat{R}_w\:J(E,L,1_E,b) \to J$ is an isotopy and the isotopic embeddings $i,i \circ R_w$ from $E$ to $J$ are easily seen to be weakly equivalent. Hence, replacing $i$ by $i \circ R_w$ and $\eta$ by $\eta \circ \hat{R}_w$ if necessary, we may assume $u = 1_E$. But then, by symmetry, we may assume $u^\prime = 1_E$ as well, forcing
\[
\eta,\eta^\prime\:J(E,L,1_E,b) \longrightarrow J
\]
to be isotopies extending $i,i^\prime$, respectively. Thus $\varphi := \eta \circ \eta^{\prime -1} \in \Str(J)$ satisfies $\varphi \circ i^\prime = \eta \circ \eta^{\prime -1} \circ \eta^\prime\vert_E = \eta\vert_E = i$, so $i$ and $i^\prime$ are strongly, hence weakly, equivalent. \end{proof} 

% % % % % % % % % % % % % % % % % % % % % % % % % % % % % % % % % % % % % % % % % % % % % % % % % % %

\section{Norm classes and strong equivalence} \label{ss.NOST}  

\subsection{} Let $(B,\tau)$ be a central simple associative algebra of degree $3$ with unitary involution over $F$ and $E$ a cubic \'etale $F$-algebra. Then the centre, $K$, of $B$ and the discriminant, $\Delta(E)$, of $E$ are quadratic \'etale $F$-algebras, as is $L := K \ast \Delta(E)$ (cf. \ref{ss.QUACU}~(a)). To any pair $(i,i^\prime)$ of isotopic embeddings from $E$ to $J := H(B,\tau)$ we will attach an invariant, belonging to $E^\times/n_L((E \otimes L)^\times)$ and called the norm class of $(i,i^\prime)$, and we will show that $i$ and $i^\prime$ are strongly equivalent if and only if their norm class is trivial. In order to achieve these objectives, a number of preparations will be needed. 

We begin with an extension of Th.~\ref{t.IMCRI} from isomorphisms to isotopies. 

\begin{prop} \label{p.ISCRI} {Let $E,E^\prime$ and $L,L^\prime$ be cubic and quadratic \'etale algebras, respectively, over $F$ and suppose we are given invertible elements $u \in E$, $u^\prime \in E^\prime$, $b \in L$, $b^\prime \in L^\prime$ satisfying $N_E(u) = n_L(b)$, $N_{E^\prime}(u^\prime) = n_{L^\prime}(b^\prime)$. We write 
\[
J := J(E,L,u,b) = E \oplus (E \otimes L)j, \quad J^\prime := J(E^\prime,L^\prime, u^\prime,b^\prime) = E^\prime \oplus (E^\prime \otimes L^\prime)j^\prime 
\]
as in} \ref{ss.DETI} {for the corresponding \'etale Tits process algebras and let $\varphi\:E^\prime \overset{\sim}\to E$ be an isotopy. Then, letting $\Phi\:J^\prime \to J$ be an arbitrary map and setting $p := \varphi(1_{E^\prime})^{-1} \in E^\times$, the following conditions are equivalent.}
\begin{enumerate}
\item \label{ISCRI.1} {$\Phi$ is an isotopy extending $\varphi$.} 

\item \label{ISCRI.2} {There exist an isomorphism $\psi\:L^\prime \overset{\sim}\to L$ and an invertible element $y \in E \otimes L$ such that $\varphi(u^\prime) = n_L(y)p^\sharp p^{-3}u$, $\psi(b^\prime) = N_E(y)b$ and
\begin{align}
 \Phi(v_0^\prime + v^\prime j^\prime) = \varphi(v_0^\prime) + \big(y(\varphi \otimes \psi)(v^\prime)\big)j
\end{align}
for all $v_0^\prime \in E^\prime$, $v^\prime \in E^\prime \otimes L^\prime$.}
\end{enumerate}
\end{prop}

\begin{proof} $\varphi_1 := R_p \circ \varphi\:E^\prime \to E$ is an isotopy preserving units, hence is an isomorphism. By \ref{ss.DETI} we have
\[
J := J(E,L,u,b) = J(L,E \otimes L,\iota_L,u,b),
\]
and in obvious notation, setting $u^{(p)} := p^\sharp u$, $b^{(p)} := N_E(p)b$, Lemma~\ref{l.SETIS} yields an isomorphism
\begin{align*}
\hat{R}_p\:J^{(p)} \overset{\sim} \longrightarrow J_1 :=\,\,&J(L,E \otimes L,\iota_L,u^{(p)},b^{(p)}) = J(E,L,u^{(p)},b^{(p)}), \\
\,\,&v_0 + vj \longmapsto (v_0p) + vj_1
\end{align*}
Thus $\hat{R}_p\:J \to J_1$ is an isotopy and $\Phi_1 := \hat{R}_p \circ \Phi$ is a map from $J^\prime$ to $J_1$. Since $\varphi_1$ preserves units, this leads to the following chain of equivalent conditions.
\begin{align*}
\text{$\Phi$ is an isotopy extending $\varphi$} \Longleftrightarrow \,\,&\text{$\Phi_1$ is an isotopy extending $\varphi_1$} \\
\Longleftrightarrow \,\,&\text{$\Phi_1$ is an isotopy extending $\varphi_1$} \\
\,\,&\text{and preserving units} \\
\Longleftrightarrow \,\,&\text{$\Phi_1$ is an isomorphism extending $\varphi_1$.}
\end{align*}
By Th.~\ref{t.IMCRI}, therefore, \eqref{ISCRI.1} holds if and only if there exist an element $y_1 \in (E \otimes L)^\times$ and an isomorphism $\psi\:L^\prime \to L$ such that  $\varphi_1(u^\prime) = n_L(y_1)u^{(p)}$, $\psi(b^\prime) = N_E(y_1)b^{(p)}$ and
\[
\Phi_1(v_0^\prime + v^\prime j^\prime) = \varphi_1(v_0^\prime) + \big(y_1(\varphi_1 \otimes \psi)(v^\prime)\big)j_1
\] 
for all $v_0^\prime \in E^\prime$, $v^\prime \in E^\prime \otimes L^\prime$. Setting $y := y_1p$, and observing $(\varphi_1 \otimes \psi)(v^\prime) = (\varphi \otimes \psi)(v^\prime)p$ for all $v^\prime \in E^\prime \otimes L^\prime$, it is now straightforward to check that the preceding equations, in the given order, are equivalent to the ones in condition \eqref{ISCRI.2} of the theorem. 
\end{proof}

\begin{lem} \emph{(\cite[Lemma~4.5]{MR2032452})} \label{l.NORNOR} {Let $L$ (resp.~$E$) be a quadratic (resp.~a cubic) \'etale $F$-algebra. Given $y \in E \otimes L$ such that $c := N_E(y)$ satisfies $n_L(c) = 1$, there exists an element $y^\prime \in E \otimes L$ satisfying $N_E(y^\prime) = c$ and  $n_L(y^\prime) = 1$.} 
\hfill $\qed$
\end{lem}

\subsection{Notation} \label{ss.NO} For the remainder of this section we fix a central simple associative algebra $(B,\tau)$ of degree $3$ with unitary involution over $F$ and a cubic \'etale $F$-algebra $E$. We write $K$ for the centre of $B$, put $J := H(B,\tau)$ and $L := K \ast \Delta(E)$ in the sense of \ref{ss.QUACU}.

\begin{thm} \label{t.EXTRI} {Let $i\:E \to J$ be an isotopic embedding and suppose $w \in E$ has norm $1$. Then the isotopic embeddings $i$ and $i \circ R_w$ from $E$ to $J$ are strongly equivalent if and only if $w \in n_L((E \otimes L)^\times)$.}
\end{thm}

\begin{proof} By Prop.~\ref{p.ITEX}, we find invertible elements $u \in E$, $b \in L$ such that $N_E(u) = n_L(b)$ and $i$ extends to an isotopy $\eta\:J_1 := J(E,L,u,b) \to J$. On the other hand, $i$ and $i \circ R_w$ are strongly equivalent by definition (cf. \ref{ss.WESO}) if and only if there exists an element $\Psi \in \Str(J)$ making the central square in the diagram
\begin{align}
\vcenter{\label{EXTRI} \xymatrix@R=0.5pc{
J_1 \ar[rrdd]_{\eta} && E \ar@{_{(}->}[ll]  \ar[rr]_{R_w} \ar[dd]_{i} &&
E \ar[dd]^{i} \,\ar@{^{(}->}[rr] && J_1 \ar[lldd]^{\eta} \\ \\
&& J \ar@{.>}[rr]_{\Psi} && J &&.}}
\end{align}
commutative, equivalently, the isotopy $\varphi := R_w\:E \to E$ can be extended to an element of the structure group of $J_1$. By Prop.~\ref{p.ISCRI} (with $p = w^{-1}$), this in turn happens if and only if some invertible element $y \in E \otimes L$ has $uw = n_L(y)(w^{-1})^\sharp w^3u = n_L(y)w^4u$, i.e., $w = n_L(w^2y)$, and either$ N_E(y) = 1$ or $N_E(y) = \bar bb^{-1}$. Replacing $y$ by $w^2y$, we conclude that $i$ and $i \circ R_w$ are strongly equivalent if and only 
\begin{align}
\label{CHEX} \text{some $y \in E \otimes L$ satisfies (i) $n_L(y) = w$ and (ii) $N_E(y) \in  \{ 1, \bar bb^{-1} \}$}. 
\end{align} 
In particular, for $i$ and $i \circ R_w$ to be strongly equivalent it is necessary that $w \in n_L((E \otimes L)^\times)$. Conversely, let this be so. Then some $y \in E \otimes L$ satisfies condition (i) of \eqref{CHEX}, so we have $w = n_L(y)$ and $n_L(N_E(y)) = N_E(n_L(y)) = N_E(w) = 1$. Hence Lemma~\ref{l.NORNOR} yields an element $y^\prime \in E \otimes L$ such that $N_E(y^\prime) = N_E(y)$ and $n_L(y^\prime) = 1$. Setting $z := yy^{\prime -1} \in E \otimes L$, we conclude $n_L(z) = n_L(y) = w$ and $N_E(z) = N_E(y)N_E(y^\prime)^{-1} = 1$, hence that \eqref{CHEX} holds for $z$ in place of $y$. Thus $i$ and $i \circ R_w$ are strongly equivalent. 
\end{proof}

\subsection{Norm classes} \label{ss.NOCLA} Let $i,i^\prime\:E \to J$ be isotopic embeddings. By Th.~\ref{t.SKONOIT}, there exist a norm-one elements $w \in E$ as well as an element $\varphi \in \Str(J)$ such that the left-hand square of the diagram
\[
\xymatrix@R=0.5pc{
E \ar[rr]_{R_w} \ar[dd]_{i^\prime} && E \ar[dd]^{i} && E \ar[ll]^{R_v} \ar[dd]^{i^\prime} \\ \\
J \ar[rr]_{\varphi} && J  && J \ar[ll]^{\psi}}
\] 
commutes. Given another norm-one element $v \in E$ and another element $\psi \in \Str(J)$ such that the right-hand square of the above diagram commutes as well, then the isotopic embeddings $i^\prime$ and $i^\prime \circ R_{v^{-1}w}$ from $E$ to $J$ are strongly equivalent (via $\psi^{-1} \circ \varphi$), and Th.~\ref{t.EXTRI} implies $w \equiv v \bmod n_L((E \otimes L)^\times$). Thus the class of $w \bmod n_L((E \otimes L)^\times)$ does not depend on the choice of $w$ and $\varphi$. We write this class as $[i,i^\prime]$ and call it the \emph{norm class} of $(i,i^\prime)$; it is clearly symmetric in $i,i^\prime$. We say $i,i^\prime$ \emph{have trivial norm class} if
\[
\text{$[i,i^\prime] = 1$ in $E^\times/n_L((E \otimes L)^\times)$.}
\]
For three isotopic embeddings $i,i^\prime,i^{\prime\prime}\:E \to J$, it is also trivially checked that $[i,i^{\prime\prime}] = [i,i^\prime][i^\prime,i^{\prime\prime}]$.

\begin{cor} \label{c.WESONO} {Two isotopic embeddings $i, i' \: E \to J$ are strongly equivalent if and only if $[i, i']$ is trivial.}
\end{cor}

\begin{proof} Let $i,i^\prime\:E \to J$ be isotopic embeddings. By Th.~\ref{t.SKONOIT}, they are weakly equivalent, so some norm-one element $w \in E$ makes $i^\prime$ and $i \circ R_w$ strongly equivalent. Thus $i$ and $i^\prime$ are strongly equivalent if and only if $i$ and $i \circ R_w$ are strongly equivalent, which by Th.~\ref{t.EXTRI} amounts to the same as $w \in n_L((E \otimes L)^\times)$, i.e., to $i$ and $i^ \prime$ having trivial norm class. \end{proof}

\begin{rmk} When confined to isomorphic rather than isotopic embeddings, Cor.~\ref{c.WESONO} reduces to \cite[Th.~4.2]{MR2032452}. \end{rmk}
  
% % % % % % % % % % % % % % % % % % % % % % % % % % % % % % % % % % % % % % % % % % % % % % % % % % %

\section{Albert algebras: proof of Theorem \ref{t.SNOTWI}} \label{s.ALAL}   

\subsection{} Unfortunately, we have not succeeded in extending  Th.~\ref{t.SKONOIT}, the notion of norm class as defined in \ref{ss.NOCLA}, or Cor.~\ref{c.WESONO} from absolutely simple Jordan algebras of degree $3$ and dimension $9$ to Albert algebras. Instead, we will have to be more modest by settling with Theorem \ref{t.SNOTWI}, i.e., with the weak Skolem-Noether property for isomorphic rather than arbitrary isotopic embeddings. Given a cubic \'etale algebra $E$ and an Albert algebra $J$ over $F$, the idea of the proof is to factor two isomorphic embeddings from $E$ to $J$ through the same absolutely simple nine-dimensional subalgebra of $J$, which by structure theory will have the form $H(B,\tau)$ for some central simple associative algebra $(B,\tau)$ of degree $3$ with unitary involution over $F$, allowing us to apply Th.~\ref{t.SKONOIT} and reach the desired conclusion. In order to carry out this procedure, a few preparations will be needed.

Throughout this section, we fix an arbitrary Albert algebra $J$ and a cubic \'etale algebra $E$ over $F$.

\begin{lem} \label{l.SPLIET} {Assume $F$ is algebraically closed and denote by $E_1 := \Diag_3(F) \subseteq \Mat_3(F)^+$ the cubic \'etale subalgebra of diagonal matrices. Then there exists a cubic \'etale subalgebra $E_2 \subseteq \Mat_3(F)^+$ such that $\Mat_3(F)^+$ is generated by $E_1$ and $E_2$ as a cubic Jordan algebra over $F$.}
\end{lem}

\begin{proof} We realize $\Mat_3(F)^+$ as a first Tits construction
\[
J_1 := \Mat_3(F)^+ = J(E_1,1),
\] 
with adjoint $\sharp$, norm $N$, trace $T$, and identify the diagonal matrices on the left with $E_1$ viewed canonically as a cubic subalgebra of $J(E_1,1)$. Since $F$ is infinite, we find an element $u_0 \in E_1$ satisfying $E_1 = F[u_0]$. Letting $\alpha \in F^\times$, we put
\begin{align*}
y := u_0 + \alpha j_1 \in J_1.
\end{align*}
Since $u_0$ and $j_1$ generate $J_1$ as a cubic Jordan algebra, so do $u_0$ and $y$, hence $E_1$ and $E_2 := F[y]$. It remains to show that, for a suitable choice of $\alpha$, the $F$-algebra $E_2$ is cubic \'etale. We first deduce from \eqref{SHAFI} and \eqref{TRAFI} that
\begin{align*}
y^\sharp =\,\,&u_0^\sharp + (-\alpha u_0)j_1 + \alpha^2j_2, \\
T(y) =\,\,&T_{E_1}(u_0), \\
T(y^\sharp) =\,\,&T_{E_1}(u_0^\sharp), \\
N(y) =\,\,&N_{E_1}(u_0) + \alpha^3.
\end{align*}
Thus $y$ has the generic minimum (= characteristic) polynomial
\[
\bft^3 - T_{E_1}(u_0)\bft^2 + T_{E_1}(u_0^\sharp)\bft - \big(N_{E_1}(u_0)+ \alpha^3\big) \in F[\bft],
\]
whose discriminant by \cite[IV, Exc.~12(b)]{MR1878556} is
\begin{align*}
\Delta_y :=\,\,&T_{E_1}(u_0)^2T_{E_1}(u_0^\sharp)^2 - 4T_{E_1}(u_0^\sharp)^3 - 4T_{E_1}(u_0)^3(N_{E_1}(u_0) + \alpha^3) \\
\,\,&- 27(N_{E_1}(u_0) + \alpha^3)^2 + 18T_{E_1}(u_0)T_{E_1}(u_0^\sharp)(N_{E_1}(u_0) + \alpha^3) \\
=\,\,&\Delta_{u_0} - \big(4T_{E_1}(u_0)^3 + 54N_{E_1}(u_0) - 18T_{E_1}(u_0)T_{E_1}(u_0^\sharp)\big)\alpha^3  - 27\alpha^6,
\end{align*}
where $\Delta_{u_0} \neq 0$ is the discriminant of the minimum polynomial of $u_0$. Regardless of the characteristic, we can therefore choose $\alpha \in F^\times$ in such a way that $\Delta_y \neq 0$, in which case $E_2$ is a cubic \'etale $F$-algebra. \end{proof}

\subsection{Digression: pointed quadratic forms} \label{ss.POQUA} By a \emph{pointed quadratic form} over $F$ we mean a triple $(V,q,c)$ consisting of an $F$-vector space $V$, a quadratic form $q\:V \to F$, with bilinearization $q(x,y) = q(x + y) - q(x) - q(y)$, and an element $c \in V$ that is a \emph{base point} for $q$ in the sense that $q(c) = 1$. Then $V$ together with the $U$-operator
\begin{align}
\label{UOP} U_xy :=q(x,\bar y)x - q(x) \bar y &&(x,y \in V), 
\end{align}
where $\bar y := q(c,y)c - y$, and the unit element $1_J := c$ is a Jordan algebra over $F$, denoted by $J := J(V,q,c)$ and called the Jordan algebra of the pointed quadratic form $(V,q,c)$. It follows immediately from \eqref{UOP} that the subalgebra of $J$ generated by a family of elements $x_i \in J$, $i \in I$, is $Fc + \sum_{i\in I} Fx_i$.

\begin{lem} \label{l.GEPO} {Assume $F$ is infinite and let $i,i^\prime\:E \to J$ be %two 
isomorphic embeddings. Then there exist isomorphic embeddings $i_1,i_1^\prime\:E \to J$ such that $i$ (resp., $i^\prime$) is strongly equivalent to $i_1$ (resp., $i_1^\prime$) and the subalgebra of $J$ generated by $i_1(E) \cup i_1^\prime(E)$ is absolutely simple of degree $3$ and dimension $9$.}
\end{lem}

\begin{proof} We proceed in two steps.  Assume first that $F$ is algebraically closed. Then $E = F \times F \times F$ and $J = \Her_3(C)$ are both split, $C$ being the octonion algebra of Zorn vector matrices over $F$. Note that $\Mat_3(F)^+ \cong \Her_3(F \times F)$ may be viewed canonically as a subalgebra of $J$. By splitness of $E$, there are frames (i.e., complete orthogonal systems of absolutely primitive idempotents) $(e_p)_{1\leq p\leq 3}$, $(e_p^\prime)_{1\leq p\leq 3}$ in $J$ such that $i(E) = \sum Fe_p$, $i^\prime(E) = \sum Fe_p^\prime$. But frames in the split Albert algebra are conjugate under the automorphism group. Hence we find automorphisms $\varphi,\psi$ of $J$ satisfying $\varphi(e_p) = \psi(e^\prime_p) = e_{pp}$ for $1 \leq p \leq 3$. Applying Lemma~\ref{l.SPLIET}, we find a cubic \'etale subalgebra $E_2 \subseteq \Mat_3(F)^+ \subseteq J$ that together with $E_1 := \Diag_3(F) = (\varphi \circ i)(E)$ generates $\Mat_3(F)^+$ as a cubic Jordan algebra over $F$. Again, the cubic \'etale $E_2$ is split, so we find a frame $(c_p)_{1\leq p\leq 3}$ in $J$ satisfying $E_2 = \sum Fc_p$. This in turn leads to an automorphism $\psi^\prime$ of $J$ sending $e_{pp}$ to $c_p$ for $1 \leq p \leq 3$. Then $i_1 := \varphi \circ i$ and $i_1^\prime := \psi^\prime \circ \psi \circ i^\prime$ are strongly equivalent to $i,i^\prime$, respectively, and satisfy $i_1(E) = E_1$, $i_1^\prime(E) = E_2$,  hence have the desired property. 

\smallskip

Now let $F$ be an arbitrary infinite field and write $\bar F$ for its algebraic closure. We have $E = F[u]$ for some $u \in E$ and put $x := i(u), x^\prime := i^\prime(u) \in J$. We write $\kalg$ for the category of commutative associative $k$-algebras with $1$, put  $G:= \Aut(J) \times \Aut(J)$ as a group scheme over $F$ and, given $R \in \kalg$, $(\varphi,\varphi^\prime) \in  G(R)$, write $x_m := x_m(\varphi,\varphi^\prime)$, $1 \leq m \leq 9$, in this order for the elements
\begin{align*}
x_1 := \,\,&1_{J_R}, \quad x_2 := \varphi(x_R), \quad x_3 := \varphi(x_R^ \sharp), \\
x_4 :=\,\,&\varphi^\prime(x_R), \quad x_5 := \varphi^\prime(x_R^\sharp), \quad x_6 := \varphi(x_R) \times \varphi^\prime(x_R), \\
x_7 :=\,\,&\varphi(x_R^\sharp) \times \varphi^\prime(x_R), \quad x_8 := \varphi(x_R) \times \varphi^\prime(x_R^\sharp), \quad x_9 := \varphi(x_R^\sharp) \times \varphi^\prime(x^\sharp).
\end{align*}
By a result of Br\"uhne (cf.~\cite[Prop.~6.6]{MR3316839}), the subalgebra of $J_R$ generated by $(\varphi \circ i_R)(E_R)$ and $(\varphi^\prime \circ i^\prime_R)(E_R)$ is spanned as an $R$-module by the elements $x_1,\dots,x_9$. Now consider the open subscheme  $X \subseteq G$ defined by the condition that  $X(R)$, $R \in \kalg$, consist of all elements $(\varphi,\varphi^\prime) \in  G(R)$ satisfying
\[
\det\Big(T_J\big(x_m(\varphi,\varphi^\prime),x_n(\varphi,\varphi^\prime)\big)\Big)_{1\leq m,n\leq 9} \in R^\times.
\]
By what we have just seen, this is equivalent to saying that the subalgebra of $J_R$ generated by $(\varphi \circ i_R)(E_R)$ and $(\varphi^\prime \circ i^\prime_R)(E_R)$ is a free $R$-module of rank $9$ and has a non-singular trace form. By the preceding paragraph, $X(\bar F) \subseteq G(\bar F)$ is a non-empty (Zariski-) open, hence dense, subset. On the other hand, by \cite[13.3.9(iii)]{Sp:LAG}, $G(F)$ is dense in  $G(\bar F)$. Hence so is  $X(F) = X(\bar F) \cap G(F)$. In particular, we can find elements $\varphi,\varphi^\prime \in  \Aut(J)(F)$ such that the subalgebra $J^\prime$ of $J$ generated by $(\varphi \circ i)(E)$ and $(\varphi^\prime \circ i^\prime)(E)$ is non-singular of dimension $9$.  This property is preserved under base field extensions, as is the property of being generated by two elements. Hence, if $J^\prime$ were not absolutely simple, some base field extension of it would split into the direct sum of two ideals one of which would be the Jordan algebra of a pointed quadratic form of dimension $8$ \cite[Th.~1]{MR0304447}. On the other hand, the property of being generated by two elements is inherited by this Jordan algebra, which by \ref{ss.POQUA} is impossible. Thus $i_1 := \varphi \circ i$ and $i_1^\prime := \varphi^\prime \circ i^\prime$ satisfy all conditions of the lemma. \end{proof}

\begin{prop} \label{p.EMNI} Let $F$ be a finite field and $i\:E \to J$ an isomorphic embedding. Writing $K := \Delta(E)$ for the discriminant of $E$, there exists a subalgebra $J_1 \subseteq J$ such that
\[
i(E) \subseteq J_1 \cong \Her_3(K,\Gamma), \quad \Gamma := \diag(1,-1,1).
\]
\end{prop}

\begin{proof} $F$ being finite, the Albert algebra $J$ is necessarily split. Replacing $E$ by $i(E)$ if necessary, we may assume $E \subseteq J$ and that $i\:E \hookrightarrow J$ is the inclusion.  We write $E^\perp \subseteq J$ for the orthogonal complement of $E$ in $J$ relative to the bilinear trace and, for all $v \in E^\perp$, denote by $q_E(v)$ the $E$-component of $v^\sharp$ relative to the decomposition $J = E \oplus E^\perp$. By \cite[Prop.~2.1]{MR85i:17029}, $E^\perp$ may be viewed as an $E$-module in a natural way, and $q_E\:E^\perp \to E$ is a quadratic form over $E$. Moreover, combining \cite[Cor.~3.8]{MR85i:17029} with a result of Engelberger \cite[Prop.~1.2.5]{En-dis}, we conclude that there exists an element $v \in E^\perp$ that is invertible in $J$ and satisfies $q_E(v) =0$. 
Now \cite[Prop.~2.2]{MR85i:17029} yields a non-zero element $\alpha \in F$ such that the inclusion $E \hookrightarrow J$ can be extended to an isomorphic embedding from the \'etale first Tits construction $J(E,\alpha)$ into $J$. Write $J_1 \subseteq J$ for its image. Then $E \subseteq J_1 \cong J(E,\alpha)$, and from \cite[Th.~3]{MR86g:17020} we deduce $J(E,\alpha) \cong \Her_3(K,\Gamma$) with $\Gamma := \diag(1,-1,1)$ as above. \end{proof}

\begin{prop} \label{p.SKONOIN}
Let $J_1,J_1^\prime$ be nine-dimensional absolutely simple subalgebras of $J$. Then every isotopy $J_1 \to J_1^\prime$ can be extended to an element of the structure group of $J$.
\end{prop}

\begin{proof} Let $\eta_1\:J_1 \to J_1^\prime$ be an isotopy. Then some $w \in  J_1^\times$ makes $\eta_1\:J_1^{(w)} \to J_1^\prime$ an isomorphism. On the other hand, structure theory yields a central simple associative algebra $(B,\tau)$ of degree $3$ with unitary involution over $F$ and an isomorphism $\varphi\:H(B,\tau) \to J_1$ which, setting $p := \varphi^{-1}(w) \in H(B,\tau)^\times$, may be regarded as an isomorphism
\[
\varphi\:H(B,\tau)^{(p)} \overset{\sim} \longrightarrow J_1^{(w)}.
\]
On the other hand, following (\ref{RIP}),
\[
R_{p}\:H(B,\tau)^{(p)} \overset{\sim} \longrightarrow H(B,\tau^{(p)})
\]
is an isomorphism as well, and combining, we end up with an isomorphism
\[
\varphi^\prime := \eta_1 \circ \varphi \circ R_p^{-1}\:H(B,\tau^{(p)}) \overset{\sim} \longrightarrow J_1^\prime.
\] 
Writing $K$ for the centre of $B$ and consulting \ref{ss.SETI}, we now find invertible elements $u \in H(B,\tau)$, $\mu \in K$ satisfying $N_B(u) = n_K(\mu)$ such that $\varphi$ extends to an isomorphism
\[
\Phi\:J(B,\tau,u,\mu) \overset{\sim} \longrightarrow J.
\]
Similarly, we find invertible elements $u^\prime \in H(B,\tau^{(p)})$, $\mu^\prime \in K$ satisfying $N_B(u^\prime) = n_K(\mu^\prime)$ such that $\varphi^\prime$ extends to an isomorphism
\[
\Phi^\prime\:J(B,\tau^{(p)},u^\prime,\mu^\prime) \overset{\sim} \longrightarrow J.
\]
Next, setting $u_1 := p^{\sharp -1}u^\prime$, $\mu_1 := N_B(p)^{-1}\mu^\prime$, we apply Lemma~\ref{l.SETIS} to obtain an isotopy
\begin{align}
\label{THERP} \hat{R}_p\:J(B,\tau,u_1,\mu_1) \longrightarrow J(B,\tau^{(p)},u^\prime,\mu^\prime), \quad v_0 + vj \longmapsto (v_0p) + (p^{-1}vp)j,
\end{align}  
and combining, we end up with an isotopy
\[
\hat{R}_p^{-1} \circ \Phi^{\prime -1} \circ \Phi\:J(B,\tau,u,\mu) \longrightarrow J(B,\tau,u_1,\mu_1).
\]
Hence \cite[Th.~5.2]{MR2063796} yields an isomorphism
\[
\Psi\:J(B,\tau,u,\mu) \overset{\sim} \longrightarrow J(B,\tau,u_1,\mu_1)
\]
inducing the identity on $H(B,\tau)$. Thus
\[
\eta := \Phi^\prime \circ \hat{R}_p \circ \Psi \circ \Phi^{-1}\:J \longrightarrow J
\]
is an isotopy that fits into the diagram 
\[
\xymatrix{J(B,\tau,u_1,\mu_1) \ar[rd]^{\hat{R}_p}  \\ 
J(B,\tau,u,\mu) \ar@/_2.5pc/[ddd]_{\Phi} \ar[u]^{\Psi}  & J(B,\tau^{(p)},u^\prime,\mu^\prime) \ar@/^2.5pc/[ddd]^{\Phi^\prime} \\ 
H(B,\tau) \ar[r]_{R_p} \,\,\ar@{^{(}->}[u] \ar[d]_{\varphi} & H(B,\tau^{(p)}) \,\,\ar@{^{(}->}[u] \ar[d]^{\varphi^\prime} \\ 
J_1 \ar[r]_{\eta_1} \,\,\ar@{^{(}->}[d] & J_1^\prime \,\,\ar@{^{(}->}[d] \\ 
J \ar[r]_{\eta} & J,}
\]
 whose arrows are either inclusions or isotopies. Now, since $\eta \circ \Phi = \Phi^\prime \circ \hat{R}_p \circ \Psi$ by definition of $\eta$, and $\hat{R}_p$ agrees with $R_p$ on $H(B,\tau)$ by \eqref{THERP}, simple diagram chasing shows that $\eta \in \Str(J)$ is an extension of $\eta_1$. \end{proof} 
 
 We can now prove Theorem \ref{t.SNOTWI} in a form reminiscent of Th.~\ref{t.SKONOIT}.

\begin{thm} \label{t.SKONOIM}
Let $J$ be an Albert algebra over $F$ and $E$ a cubic \'etale $F$-algebra. Then the pair $(E,J)$ satisfies the weak Skolem-Noether property for isomorphic embeddings.
\end{thm}

\begin{proof} Leit $i,i^ \prime\:E \to J$ be two isomorphic embeddings. We must show that they are weakly equivalent and first claim that we may assume the following: \emph{there exist a central simple associative algebra $(B,\tau)$ of degree $3$ with unitary involution over $F$ and a subalgebra $J_1 \subseteq J$ such that $J_1 \cong H(B,\tau)$ and $i,i^\prime$ factor uniquely through $J_1$ to isomorphic embeddings $i_1\:E \to J_1$, $i_1^\prime\:E \to J_1$.} Indeed, replacing the isomorphic embeddings $i,i^\prime$ by strongly equivalent ones if necessary, this is clear by Lemma~\ref{l.GEPO} provided $F$ is infinite. On the other hand, if $F$ is finite, Prop.~\ref{p.EMNI} leads to absolutely simple nine-dimensional subalgebras $J_1,J_1^\prime \subseteq J$ that are isomorphic and have the property that $i,i^\prime$ factor uniquely through $J_1,J_1^\prime$ to isomorphic embeddings $i_1\:E \to J_1$, $i_1^\prime\:E \to J_1^\prime$, respectively. But every isomorphism from $J_1^\prime$ to $J_1$ extends to an automorphism of $J$ \cite[40.15]{KMRT}, \cite[Remark~5.6(b)]{MR2063796}. Hence we may assume $J_1^\prime = J_1$, as desired.

With $J_1,i_1,i_1^\prime$ as above, Th.~\ref{t.SKONOIT} yields elements $w \in E$ of norm $1$ and $\varphi_1 \in \Str(J_1)$ such that  $i_1 \circ R_w = \varphi_1 \circ i_1^\prime$. Using Prop.~\ref{p.SKONOIN}, we extend $\varphi_1$ to an element $\varphi \in \Str(J)$ and therefore conclude that the diagram (\ref{ISTEM}) commutes. \end{proof}

%%%%%%%%%%%%%%%%%%%%%%%%%%%%%%%%%%%%%%%%%%
\section{Outer automorphisms for type $^3D_4$: proof of Theorem \ref{MT.D4}}

In this section, we apply Theorem \ref{t.SNOTWI} to prove Theorem \ref{MT.D4}.

%\comment Reorganized the next few sentences. \endcomment
\subsection{A subgroup of $\Str(J)$}\label{subgrp.Str}
Let $E$ be a cubic \'etale subalgebra of an Albert algebra $J$ and %put 
 write $H$ for the subgroup of $h \in \Str(J)$ that normalize $E$ and such that $Nh = N$.  Note that, for $\varphi \in \Aut(E)$, the element $\psi %= \varphi \circ R_w  \in \Str(J)$ 
 \in \Str(J)$ provided by Theorem \ref{t.SNOTWI} to extend $\varphi \circ R_w$ to all of $J$  belongs to $H$. Indeed, as $\psi \in \Str(J)$, there is a $\mu \in \Fx$ such that $N\psi = \mu N$, but for $e \in E$ we have $N(\psi(e)) = N(\varphi(ew)) = N(\varphi(e)) N(\varphi(w)) = N(e)$.

We now describe $H$ in the case where $J$ is a matrix Jordan algebra as in \S\ref{ss.CUMA} with $\Gamma =  \Eins_3$ and $E$ is the subalgebra of diagonal matrices.  We rely on some facts that are only proved in the literature under the hypothesis $\car F \ne 2, 3$.  This hypothesis is not strictly necessary but we adopt it for now in order to ease the writing. Fix $h \in H$.  The norm $N$ restricts to $E$ as $N(\sum \alpha_i e_{ii}) = \alpha_1 \alpha_2 \alpha_3$, so $h$ permutes the three singular points $[e_{ii}]$ in the projective variety $N\vert_E = 0$ in $\mathbb{P}(E)$.  There is an embedding of the symmetric group on 3 letters, $\Sym_3$, in $H$ acting by permuting the $e_{ii}$ by their indices, see \cite[\S3.2]{G:ur} for an explicit formula, and consequently $H \cong H_0 \rtimes \Sym_3$, where $H_0$ is the subgroup of $H$ of elements normalizing $Fe_{ii}$ for each $i$.  For $w := (w_1, w_2, w_3) \in (\Fx)^{\times 3}$ such that $w_1 w_2 w_3 = 1$, 
it follows that $U_w \in H$ (cf. \eqref{NOU}) sends $e_{ii} \mapsto w_i^2e_{ii}$. 

Assuming now that $F$ is algebraically closed, after multiplying $h$ by a suitable  $U_w$, we may assume that $h$ restricts to be the identity on $E$.  The subgroup of such elements of $\Str(J)$ is identified with the $\Spin(C)$ which acts on the off-diagonal entries in $J$ as a direct sum of the three inequivalent minuscule 8-dimensional representations, see \cite[36.5, 38.6, 38.7]{KMRT} or \cite[p.~18, Prop.~6]{Jac:ex}. Thus, we may identify $H$ with 
$(R^{(1)}_{E/F}(\Gm) \cdot \Spin(C)) \rtimes \Sym_3$, where $\Sym_3$ acts via outer automorphisms on $\Spin(C)$ as in \cite[\S3]{G:ur} or \cite[35.15]{KMRT}.

\subsection{The Tits class} Recall that the Dynkin diagram of a group $G$ is endowed with an action by the absolute Galois group of $F$, and elements of $\aut(\D)(F)$ act naturally on $H^2(F, Z)$.

\begin{lem} \label{tG.zero}
Let $G$ be a group of type $D_4$ over a field $F$ with Dynkin diagram $\D$.  If there is a $\pi \in \Aut(\D)(F)$ of order 3 such that $\pi(t_G) = t_G$, then $G$ has type $^1D_4$ or $^3D_4$ and $t_G = 0$.
\end{lem}

\begin{proof}
For the first claim, if $G$ has type $^2D_4$ or $^6D_4$, then $\Aut(\D)(F) = \Z/2$ or $1$.

Now suppose that $G$ has type $^1D_4$.   We may assume that $G$ is simply connected.  The center $Z$ of the simply connected cover of $G$ is $\mu_2 \times \mu_2$, with automorphism group $\Sym_3$ and $\pi$ acts on $Z$ with order 3. 
The three nonzero characters $\chi_1, \chi_2, \chi_3 \!: Z \to \Gm$ are permuted transitively by $\pi$, so by hypothesis the element $\chi_i(t_G) \in H^2(F, \Gm)$ does not depend on $i$.  As the $\chi_i$'s satisfy the equations $\chi_1 + \chi_2 + \chi_3 = 0$ and $2\chi_i = 0$ (compare \cite[6.2]{Ti:R} or \cite[9.14]{KMRT}), it follows that $\chi_i(t_G) = 0$ for all $i$, hence $t_G = 0$ by \cite[Prop.~7]{G:outer}.

If $G$ has type $^3D_4$, then there is a unique cyclic cubic field extension $E$ of $F$ such that $G \times E$ has type $^1D_4$.  By the previous paragraph, restriction $H^2(F,Z) \to H^2(E,Z)$ kills $t_G$.  That map is injective because $Z$ has exponent 2, so $t_G = 0$.
\end{proof}

In the next result, the harder, ``if'' direction is the crux case of the proof of Theorem \ref{MT.D4} and is an application of Theorem \ref{t.SNOTWI}.  The easier, ``only if'' direction amounts to \cite[Th.~13.1]{ChEKT} or \cite[Prop.~4.2]{KT:3D4}; we include it here as a consequence of the (a priori stronger) Lemma \ref{tG.zero}.

\begin{prop} \label{tG.iff}
Let $G$ be a group of type $D_4$ over a field $F$.  The image of $\alpha(F) \!: \Aut(G)(F) \to \Aut(\D)(F)$ contains an element of order 3 if and only if $G$ has type $^1D_4$ or $^3D_4$, $G$ is simply connected or adjoint, and $t_G = 0$.  
\end{prop}

\begin{proof}
\emph{\underline{``If'' }}:
We may assume that $G$ is simply connected.  If $G$ has type $^1D_4$, then $G$ is $\Spin(q)$ for some 3-Pfister quadratic form $q$, and the famous triality automorphisms of $\Spin(q)$ as in \cite[3.6.3, 3.6.4]{SpV} are of order 3 and have image in $\aut(\D)(F)$ of order 3.  So assume $G$ has type $^3D_4$.

Assume for this paragraph that $\car F \ne 2, 3$.  There is a uniquely determined cyclic Galois field extension $E$ of $F$ such that $G \times E$ has type $^1D_4$.  
By hypothesis, there is an Albert $F$-algebra $J$ with norm form $N$ such that $E \subset J$ and we may identify $G$ with the algebraic group with $K$-points
\[
\{ g \in \GL(J \ot K) \mid \text{$Ng = N$ and $g\vert_{E \ot K} = \Id_{E \ot K}$} \}
\]
for every extension $K$ of $F$. Take now $\varphi$ to be a non-identity $F$-automorphism of $E$ and $w \in E$ of norm $1$ and $\psi \in \Str(J)$ to be the elements given by Theorem \ref{t.SNOTWI} such that $\psi\vert_E = \varphi \circ R_w$.  As $\psi$ normalizes $E$ and preserves $N$, it follows immediately that $\psi$ normalizes $G$ as a subgroup of $\Str(J)$.  (Alternatively this is obvious from the fact that in subsection \ref{subgrp.Str}, $\Spin(C)$ is the derived subgroup of $H^\circ$.)  Tracking through the description of $H$ in subsection \ref{subgrp.Str}, we find that conjugation by $\psi$ is an outer automorphism of $G$ such that $\psi^3$ is inner.

\newcommand{\cG}{\mathcal{G}}
\newcommand{\Gt}{\widetilde{G}}

In case $F$ has characteristic 2 or 3, one can reduce to the case of characteristic zero as follows.  Find $R$ a complete discrete valuation ring with residue field $F$ and fraction field $K$ of characteristic zero.  Lifting $E$ to $R$ allows us to construct a quasi-split simply connected group scheme $\cG^q$ over $R$ whose base change to $F$ is the quasi-split inner form $G^q$ of $G$.  We have maps
\[
H^1(F, G^q) \xleftarrow{\sim} H^1_{\text{\'et}}(R, \cG^q) \hookrightarrow H^1(K, \cG^q \times K)
\]
where the first map is an isomorphism by Hensel and the second map is injective by \cite{BrTi3}.  Twisting by a well chosen $\cG^q$-torsor, we obtain
\[
H^1(F, G) \xleftarrow{\sim} H^1_{\text{\'et}}(R, \cG) \hookrightarrow H^1(K, \cG \times K)
\]
where $\cG \times K$ has type $^3D_4$ and zero Tits class and $G \cong \cG \times F$.  Now in $\Aut(G)(F) \to \Aut(\D)(F) = \Z/3$, the inverse image of 1 is a connected component $X$ of $\Aut(G)$ defined over $F$, a $G$-torsor.  Lifting $X$ to $H^1(K, \cG \times K)$, we discover that this $G$-torsor is trivial (by the characteristic zero case of the theorem), hence $X$ is $F$-trivial, i.e., has an $F$-point.
  
  \smallskip
  \emph{\underline{ ``Only if'' }}: Let $\phi \in \Aut(G)(F)$ be such that $\alpha(\phi)$ has order 3.  In view of the inclusion \eqref{alpha}, Lemma \ref{tG.zero} applies.  If $G$ has type $^3D_4$, then it is necessarily simply connected or adjoint, so assume $G$ has type $^1D_4$.  Then $\phi$ lifts to an automorphism of the simply connected cover $\Gt$ of $G$, hence acts on the center $Z$ of $\Gt$ in such a way that it preserves the kernel of the map $Z \to G$.  As $Z$ is isomorphic to $\mu_2 \times \mu_2$ and $\phi$ acts on it as an automorphism of order 3, the kernel must be 0 or $Z$, hence $G$ is simply connected or adjoint.
\end{proof}

\subsection{Proof of Theorem \ref{MT.D4}}
Let $G$ be a group of type $^3D_4$, so $\aut(\D)(F) = \Z/3$; put $\pi$ for a generator.  If $\pi(t_G) \ne t_G$, then the right side of \eqref{alpha} is a singleton and the containment is trivially an equality, so assume $\pi(t_G) = t_G$.  Then $t_G = 0$ by Lemma \ref{tG.zero} and the conclusion follows by Proposition \ref{tG.iff}.
\hfill $\qed$

\begin{eg} \label{KT.counter}
Let $F_0$ be a field with a cubic Galois extension $E_0$.  For the split adjoint group $\PSO_8$ of type $D_4$ over $F$, a choice of pinning gives an isomorphism of $\Aut(\PSO_8)$ with $\PSO_8 \rtimes \Sym_3$ where $\Sym_3$ denotes the symmetric group on 3 letters, such that elements of $\Sym_3$ normalize the Borel subgroup appearing in the pinning.  Twisting $\Spin_8$ by a 1-cocycle with values in $H^1(F_0, \Sym_3)$ representing the class of $E_0$ gives a simply connected quasi-split group $G^q$ of type $^3D_4$.  As in \cite[pp.~11, 12]{GMS}, there exists an extension $F$ of $F_0$ and a versal torsor $\xi \in H^1(F, G^q)$; define $G$ to be $G^q \times F$ twisted by $\xi$.  As $\xi$ is versal, the Rost invariant $r_{G^q}(\xi) \in H^3(F, \Z/6\Z)$ has maximal order, namely 6 \cite[p.~149]{GMS}.  Moreover, the map $\alpha(F) \!: \Aut(G)(F) \to \Aut(\Delta)(F) = \Z/3$ is onto by Theorem \ref{MT.D4}. 

In case $\car F_0 \ne 2, 3$, $G$ is $\Aut(\Gamma)$ for some twisted composition $\Gamma$ in the sense of \cite[\S36]{KMRT}.   As $r_{G^q}(\xi)$ is not 2-torsion, by \cite[40.16]{KMRT}, $\Gamma$ is not Hurwitz, and by \cite{KT:3D4}, $\aut(G)(F)$ contains no outer automorphisms of order 3.  This is a newly observed phenomenon, in that in all other cases where $\alpha(F)$ is known to be onto, it is also split.
\end{eg}

%%%%%%%%%%%%%%%%%%%%%%%%%%%%%%%%%%%%%%%%%
\section{Outer automorphisms for type $A$}

\subsection{Groups of type $A_n$} We now consider Conjecture \ref{outer.conj} and Question \ref{question.refined} for groups $G$ of type $A_n$.  If $G$ has inner type (i.e., is isogenous to $\SL_1(B)$ for a degree $d$ central simple $F$-algebra) then equality holds in \eqref{alpha} and the answer to Question \ref{question.refined} is ``yes'' as in \cite[p.~232]{G:outer}.  

So assume that $G$ has outer type and in particular $n \ge 2$.  The simply connected cover of $G$ is  $\SU\Bt$ for $B$ a central simple $K$-algebra of degree $d := n+1$, where $K$ is a quadratic \'etale $F$-algebra, and $\tau$ is a unitary $K/F$-involution.  (This generalizes the $(K, B, \tau)$ defined in \S\ref{ss.ASIN} by replacing 3 by $d$.)  As the center $Z$ of $\SU\Bt$ is the group scheme $(\mu_d)_{[K]}$ of $d$-th roots of unity twisted by $K$ in the sense of \cite[p.~418]{KMRT} (i.e., is the Cartier dual of the finite \'etale group scheme $(\Z/d)_{[K]}$), every subgroup of $Z$ is characteristic, hence \eqref{alpha} is an equality for $G$ if and only if it is so for $\SU\Bt$ and similarly the answers to Question \ref{question.refined} are the same for $G$ and $\SU\Bt$.  Therefore, we need only treat $\SU\Bt$ below.

The automorphism group $\Aut(\D)(F)$ is $\Z/2$ and its nonzero element $\pi$ acts on $H^2(F, Z)$ as $-1$, hence $\pi(t_{\SU\Bt}) = -t_{\SU\Bt}$ and the right side of \eqref{alpha} is a singleton (if $2t_{\SU\Bt} \ne 0$) or has two elements (if $2t_{\SU\Bt} = 0$).  These cases are distinguished by the following lemma.

\begin{lem} \label{tG.A} In case $d$ is even (resp., odd):
$2t_{\SU\Bt} = 0$ if and only if $B \ot_K B$ (resp., $B$) is a matrix algebra over $K$.
\end{lem}

\begin{proof}
The cocenter $Z^* := \Hom(Z, \Gm)$ is $(\Z/d)_{[K]}$; put $\chi_i \in Z^*$ for the element corresponding to $i \in (\Z/d)_{[K]}$.  If $d = 2e$ for some integer $e$, then the element $\chi_e$ is fixed by $\Gal(F)$ and $2\chi_e = \chi_d = 0$, regardless of $B$ or $t_{\SU\Bt}$.  All other $\chi_i$ have stabilizer subgroup $\Gal(K)$ and $\chi_i(2t_{\SU\Bt}) \in H^2(K, Z)$ can be identified with the class of $B^{\ot 2i}$ in the Brauer group of $K$, cf.~\cite[p.~378]{KMRT}.

The algebra $B \ot_K B$ is a matrix algebra, then, if and only if $\chi_i$ vanishes on $2t_{\SU\Bt}$ for all $i$.  This is equivalent to $2t_{\SU\Bt} = 0$ by \cite[Prop.~7]{G:outer}.  When the degree $d$ of $B$ is odd, $B \ot_K B$ is a matrix algebra if and only if $B$ is such.
\end{proof}

\begin{cor} \label{conj.A}
If $G$ is a group of type $A_n$ for $n$ even, then equality holds in \eqref{alpha} and the answer to Question \ref{question.refined} is ``yes''.
\end{cor}

\begin{proof}
We may assume that $G$ has outer type and is $\SU\Bt$.
If $2t_{\SU\Bt} \ne 0$, then the right side of \eqref{alpha} is a singleton and the claim is trivial.  Otherwise, by Lemma \ref{tG.A}, $B$ is a matrix algebra, i.e., $\SU\Bt$ is the special unitary group of a $K/F$-hermitian form, and the claim follows.
\end{proof}

\subsection{}  The algebraic group $\Aut(\SU\Bt)$ has two connected components: the identity component, which is identified with the adjoint group of $\SU\Bt$, and the other component, whose  $F$-points  are the outer automorphisms of $\SU\Bt$. 

\begin{thm} \label{A.weak}
There is an isomorphism between the $F$-variety of $K$-linear anti-automorphisms of $B$ commuting with $\tau$ and the non-identity component of $\SU\Bt$, given by sending an anti-automorphism $\psi$ to the outer automorphism $g \mapsto \psi(g)^{-1}$.
\end{thm}

Clearly, such an anti-automorphism provides an isomorphism of $B$ with its opposite algebra, hence can only exist when $B$ has exponent 2.  This is a concrete illustration of the inclusion \eqref{alpha}.

\begin{proof}
First suppose that $F$ is separably closed, in which case we may identify $K = F \times F$,  $B = M_d(F) \times M_d(F)$, and $\tau(b_1, b_2) = (b_2^t, b_1^t)$.  A $K$-linear anti-automorphism $\psi$ is, by Skolem-Noether, of the form $\psi(b_1, b_2) = (x_1 b_1^t x_1^{-1}, x_2 b_2^t x_2^{-1})$ for some $x_1, x_2 \in \PGL_d(F)$, and the assumption that $\psi\tau=\tau\psi$ forces that $x_2 = x_1^{-t}$.

As $\Nrd_{B/K}\psi = \Nrd_{B/K}$, it follows that $\psi$ is an automorphism of the variety $\SU\Bt$, hence $\phi$ defined by $\phi(g) := \psi(g)^{-1}$ is an automorphism of the group.  As $\phi$ acts nontrivially on the center --- $\phi(b) = b^{-1}$ for $b \in K^\times$ --- $\phi$ is an outer automorphism.

We have shown that there is a well-defined morphism from the variety of anti-automorphisms commuting with $\tau$ to the outer automorphisms of $\SU\Bt$, and it remains to prove that it is an isomorphism.  For this, note that $\PGL_d$ acts on $\SU\Bt$ where the group action is just function composition, that this action is the natural action of the identity component of $\SU\Bt$ on its other connected component, and that therefore the outer automorphisms are a $\PGL_d$-torsor.  Furthermore, the first paragraph of the proof showed that the anti-automorphisms commuting with $\tau$ also make up a $\PGL_d$-torsor, where the actions are related by $y.\psi = y^{-1}.\phi$ for $y \in \PGL_d$.  This completes the proof for $F$ separably closed.

For general $F$, we note that the map $\psi \mapsto \phi$ is $F$-defined and gives an isomorphism over $\Fsep$, hence is an isomorphism over $F$.
\end{proof}

\subsection{}
We do not know how to prove or disprove existence of an anti-automorphism commuting with $\tau$ in general, but we can give a criterion for Question \ref{question.refined} that is analogous to the one given in \cite{KT:3D4} for groups of type $^3D_4$.

\begin{cor} \label{A.strong}
A group $\SU\Bt$ of outer type $A$ has an $F$-defined outer automorphism of order 2 if and only if there exists a central simple algebra $(B_0, \tau_0)$ over $F$ with $\tau_0$ an involution of the first kind such that $(B, \tau)$ is isomorphic to $(B_0 \ot K, \tau_0 \ot \iota)$, for $\iota$ the non-identity $F$-automorphism of $K$.
\end{cor}

\begin{proof}
The bijection in Theorem \ref{A.weak} identifies outer automorphisms of order 2 with anti-automorphisms of order 2.  If such a $(B_0, \tau_0)$ exists, then clearly $\tau_0$ provides an anti-automorphism of order 2.

Conversely, given an anti-automorphism $\tau_0$ of order 2, define a semilinear automorphism of $B$ via $\iota := \tau_0 \tau$.  Set $B_0 := \{ b \in B \mid \iota(b) = b \}$; it is an $F$-subalgebra and $\tau_0$ restricts to be an involution on $B_0$.
\end{proof}

\begin{eg}\label{A.no}
We now exhibit a $\Bt$ with $B$ of exponent 2, but such that $\SU\Bt$ has no outer automorphism of order 2 over $F$.  The paper \cite{ART} provides a field $F$ and a division $F$-algebra $C$ of degree 8 and exponent 2 such that $C$ is not a tensor product of quaternion algebras.  Moreover, it provides a quadratic extension $K/F$ contained in $C$.  It follows that $C \ot K$ has index 4, and we set $B$ to be the underlying division algebra.  As $\cores_{K/F}[B] = 2[C] = 0$ in the Brauer group, $B$ has a unitary involution $\tau$.

For sake of contradiction, suppose that $\SU\Bt$ had an outer automorphism of order 2, hence there exists a $(B_0, \tau_0)$ as in Corollary \ref{A.strong}.  Then $B_0$ has degree 4, so $B_0$ is a biquaternion algebra.  Moreover, $C \ot B_0$ is split by $K$, hence is Brauer-equivalent to a quaternion algebra $Q$.  By comparing degrees, we deduce that $C$ is isomorphic to $B_0 \ot Q$, contradicting the choice of $C$.
\end{eg}

\subsection{Type $^2E_6$}
Results entirely analogous to  Theorem \ref{A.weak}, Corollary \ref{A.strong}, and Example \ref{A.no} also hold  for groups $G$ of type $^2E_6$, using proofs of a similar flavor.  The Dynkin diagram of type $E_6$ has automorphism group $\Z/2 = \{ \Id, \pi \}$, and arguing as in Lemmas \ref{tG.zero} or \ref{tG.A} shows that $\pi(t_G) = t_G$ if and only if $t_G = 0$.  So for addressing Conjecture \ref{outer.conj} and Question \ref{question.refined}, it suffices to consider only those groups with zero Tits class, which can be completely described in terms of the hermitian Jordan triples introduced in \cite[\S4]{GPe} or the Brown algebras studied in \cite{G:struct}.  We leave the details to the interested reader.

Does Conjecture \ref{outer.conj} hold for every group of type $^2E_6$?  One might hope to imitate the outline of the proof of Theorem \ref{MT.D4}.  Does an analogue of Theorem \ref{t.SNOTWI} hold, where one replaces Albert algebras, cubic Galois extensions, and the inclusion of root systems $D_4 \subset E_6$ by Brown algebras or Freudenthal triple systems, quadratic Galois extensions, and the inclusion $E_6 \subset E_7$?

{\small\subsection*{Acknowledgements}
The first author was partially supported by NSF grant DMS-1201542. 
He thanks TU Dortmund for its hospitality under the auspices of the Gambrinus Fellowship while some of the research was conducted.  The second author thanks the Department of Mathematics and Statistics of the University of Ottawa, where much of this research was initiated, for its support and hospitality.}

\def\cprime{$'$}
\providecommand{\bysame}{\leavevmode\hbox to3em{\hrulefill}\thinspace}
\providecommand{\MR}{\relax\ifhmode\unskip\space\fi MR }
\providecommand{\MRhref}[2]{%
  \href{http://www.ams.org/mathscinet-getitem?mr=#1}{#2}
}
\providecommand{\href}[2]{#2}

\end{document}